\documentclass[11pt]{article}
\usepackage[margin=1in]{geometry}
\usepackage[utf8]{inputenc}
\usepackage{amsmath,amssymb,amsthm}
\usepackage{mathtools}
\usepackage{graphicx}
\usepackage{booktabs}
\usepackage{xcolor}
\usepackage{algorithm}
\usepackage{algorithmic}
\usepackage{enumitem}
\usepackage{hyperref}

\definecolor{myBlue}{rgb}{0.0,0.0,0.55}
\hypersetup{colorlinks=true,citecolor=myBlue,linkcolor=myBlue,urlcolor=myBlue}

\theoremstyle{plain}
\newtheorem{theorem}{Theorem}[section]
\newtheorem{lemma}[theorem]{Lemma}
\newtheorem{corollary}[theorem]{Corollary}
\newtheorem{proposition}[theorem]{Proposition}
\theoremstyle{definition}

\newtheorem{example}[theorem]{Example}
\theoremstyle{remark}
\newtheorem{remark}[theorem]{Remark}

\newcommand{\email}[1]{\href{mailto:#1}{\texttt{#1}}}

\setlist[enumerate]{leftmargin=20pt}
\setlist[itemize]{leftmargin=18pt}
\setlist[enumerate,1]{label=\arabic*.}

\newcommand{\one}{\mathbf{1}}
\newcommand{\dom}{\operatorname{dom}}

\DeclareMathOperator*{\argmin}{arg\,min}
\begin{document}
\title{Acceleration of Generalized Conditional Gradient Methods}
\author{Long Chen\thanks{Department of Mathematics, University of California, Irvine, CA 92697, USA (\email{chenlong@math.uci.edu}).} \and Zeyi Xu\thanks{Department of Mathematics, University of California, Irvine, CA 92697, USA (\email{zeyix1@uci.edu}).}}
\date{}
\maketitle
\begin{abstract}
We develop an adaptive acceleration of generalized conditional gradient methods for convex objectives with a strongly convex regularizer. A computable primal--dual residual certifies Lyapunov decrease without evaluating the regularizer or any Hessian. Each iteration accepts a balanced accelerated trial when the residual is nonpositive; otherwise it takes a mirror-descent step with doubling backtracking and sets the auxiliary parameter explicitly from the dual decrease and the stopping gap. Under uniform strong-convexity and smoothness bounds on the relevant sublevel set, the method converges globally at a linear rate. Near a nondegenerate minimizer, balanced trials pass above a fixed curvature threshold. The adaptive curvature updates then guarantee a positive fraction of accelerated steps and a local accelerated rate. The method applies directly to entropically regularized Jordan--Kinderlehrer--Otto (JKO) steps whose energies contain nonlocal interaction terms with nondiagonal Hessians. Experiments on JKO steps for granular media and for Poisson--Nernst--Planck systems, in one and three dimensions, show that the method reduces both oracle calls and computing time by a large fraction compared with generalized conditional gradient methods with fixed or adaptive step sizes.


\end{abstract}

\section{Introduction}
We study the regularized convex problem
\begin{equation}\label{intro:reg}
\min_{x\in\mathbb R^m}F_\sigma(x),\qquad F_\sigma(x):=F(x)+T_\sigma(x),
\end{equation}
where $F$ is Legendre and $T_\sigma$ is strongly convex. We assume that $F$, $\nabla F$, $T_\sigma^*$, and $\nabla T_\sigma^*$ are available. This structure occurs in regularized inverse problems and entropically regularized transport.

The generalized conditional gradient (GCG) method \cite{BLM} linearizes $F$ and retains the regularizer. Its relaxed update is
\begin{equation}\label{intro:relaxed}
x_{k+1}=x_k+\eta_k(x_k^{++}-x_k),\qquad 0<\eta_k\leq1,
\end{equation}
where the oracle point $x_k^{++}=\nabla T_\sigma^*(-\nabla F(x_k))$ minimizes $\langle\nabla F(x_k),x\rangle+T_\sigma(x)$. Relaxation is needed because the full oracle step need not decrease the objective.

The mirror-descent interpretation of GCG is established in \cite{Bach,Pena}. Fenchel duality gives $\min_x F_\sigma(x)=-\min_v f(v)$, where
\begin{equation}\label{eq:dualObjective}
f(v):=F_-^*(v)+T_\sigma^*(v),\qquad F_-^*(v):=F^*(-v).
\end{equation}
With $v_k=-\nabla F(x_k)$ and $x_k=\nabla F^*(-v_k)=-\nabla F_-^*(v_k)$, the relaxed update \eqref{intro:relaxed} becomes mirror descent for $f$ with mirror function $F_-^*$:
\begin{equation}\label{eq:mirror}
\nabla F_-^*(v_{k+1})=\nabla F_-^*(v_k)-\eta_k\nabla f(v_k).
\end{equation}
The dual objective is relatively strongly convex with modulus one because $f=F_-^*+T_\sigma^*$. Relative smoothness means $D_f(u,v)\leq L D_{F_-^*}(u,v)$. Mirror descent with $\eta_k=1/L$ contracts the Bregman distance to the solution by the factor $1-1/L$ per step \cite{LFN}. We choose $L_k$ adaptively and accelerate this oracle iteration.

Write $g(x)=x-x^{++}$. A Hessian-driven Nesterov accelerated gradient (HNAG) splitting with relative-convexity modulus one \cite{CL,ChenHaoWei2025Accelerated} motivates the trial updates
\begin{equation}\label{intro:acc-scheme}
\begin{aligned}
x_{k+1}-x_k&=\alpha_k(y_k-x_{k+1})-\frac{1}{L_k}g(x_k),\\
y_{k+1}-y_k&=\alpha_k(x_{k+1}^{++}-y_{k+1}).
\end{aligned}
\end{equation}

We use the primal--dual energy $\mathcal E(y,v)=F_\sigma(y)+f(v)$ and prove the one-step certificate
$$
\begin{aligned}
(1+\alpha_k)\mathcal E(y_{k+1},v_{k+1})
&\leq\mathcal E(y_k,v_k)+\mathcal R_k,
\end{aligned}
$$
where the computable residual is
$$
\mathcal R_k
=f(v_{k+1})-f(v_k) +(1+\alpha_k)D_F(y_{k+1},x_{k+1})-D_F(y_k,x_{k+1}).
$$
The certificate follows from the $y$-update for any admissible primal trial. 

Each iteration evaluates a trial with $\alpha_k=1/\sqrt{L_k}$ and accepts it if $\mathcal R_k\leq0$. When the trial is certified, it gives the accelerated contraction factor $(1+1/\sqrt{L_k})^{-1}$. Otherwise, it takes the GCG step $x_{k+1}=x_k-L_k^{-1}g(x_k)$, doubling $L_k$ until the relative-smoothness test passes. It then sets $\alpha_k=-\Delta f_k/D_F(x_{k+1}^{++},x_{k+1})$, where $\Delta f_k=f(v_{k+1})-f(v_k)$, and recomputes only $y_{k+1}$. This choice makes the updated residual nonpositive. We call this method adaptive HNAG (AHNAG). 

Under the uniform bounds of Theorem~\ref{thm:aod-global}, this gives global linear convergence without counting accelerated steps. Near a nondegenerate minimizer, balanced trials pass above a fixed curvature threshold. The curvature updates then ensure a positive fraction of accelerated steps and a local accelerated rate.

For mirror-descent methods, acceleration does not follow from relative smoothness alone \cite{DTdB}. Accelerated Bregman methods use additional geometry, such as triangle scaling \cite{HRX} or GCS conditions~\cite{ChenAccMirror,XC-aamd}. Mirror duality also yields acceleration of gradient measures \cite{kim2023mirror}, while local conditional-gradient acceleration \cite{Carderera2021Local} and fully corrective GCG \cite{Bredies2024FullyCorrective} use different oracle or correction structures. Building on the established mirror interpretation of GCG \cite{Bach,Pena}, our contribution is the adaptive two-state oracle method, its computable decrease certificate, and the conditions that yield accelerated steps.

The numerical examples are entropically regularized Jordan--Kinderlehrer--Otto (JKO) steps \cite{JKO1998,Peyre,Carlier}. The JKO method constructs Wasserstein gradient flows through successive minimization problems, so faster inner solves reduce the cost of computing the flow. In the granular-media example, AHNAG oracle calls grow roughly like $\sqrt{L_\star}$, while those of adaptive GCG grow nearly linearly in $L_\star$. In the Poisson--Nernst--Planck examples, AHNAG uses fewer calls than adaptive GCG whenever both reach the tolerance. These counts include rejected trials and measure the gain beyond curvature adaptation.

Section~\ref{sec:dual-mirror} recalls the oracle and dual geometry. Section~\ref{sec:discrete-convergence} introduces the HNAG flow, the discrete certificate, and the adaptive method. Section~\ref{sec:convergence} proves global linear and local accelerated convergence. Section~\ref{sec:numerics} presents the numerical results, and Section~\ref{sec:conclusion} concludes the paper.

\section{Dual mirror descent for regularized convex problems}\label{sec:dual-mirror}
We recall the generalized conditional gradient method \cite{BLM} and its known mirror-descent interpretation~\cite{Bach,Pena}.

\subsection{Regularized convex optimization}
Let $F:\mathbb R^m\to\mathbb R\cup\{+\infty\}$ be a proper, lower semicontinuous Legendre function, with $\Omega_F:=\operatorname{int}\dom F\neq\varnothing$. Let $T_\sigma$ be proper, lower semicontinuous, and convex, with a fixed parameter $\sigma$. Define
$
\Omega_T:=\operatorname{ri}\dom T_\sigma,
\Omega_{F_\sigma}:=\Omega_F\cap\Omega_T,
$
and assume that $\Omega_{F_\sigma}$ is nonempty and $T_\sigma$ is strongly convex. We study the regularized convex problem
\begin{equation}\label{eq:reg}
\min_{x\in\mathbb R^m}F_\sigma(x),\qquad F_\sigma(x):=F(x)+T_\sigma(x).
\end{equation}
We assume that the conjugate
{$$
  T_\sigma^*(v):=\sup_{x\in\mathbb R^m}
  \{\langle v,x\rangle-T_\sigma(x)\}
$$}
has a computable value and gradient. Strong convexity makes $T_\sigma^*$ finite and differentiable.

Assume that \eqref{eq:reg} has a minimizer $x^\star\in\Omega_{F_\sigma}$. Strict convexity of $F_\sigma$ makes it unique. Convex optimality and conjugacy give
\begin{equation}\label{eq:regEuler}
  x^\star=\nabla T_\sigma^*(-\nabla F(x^\star)).
\end{equation}
The minimizer depends on $\sigma$; we suppress this dependence in $x^\star$.

\subsection{Fixed-point and conditional-gradient forms}
{The optimality condition \eqref{eq:regEuler} suggests the fixed-point iteration
\begin{equation}\label{eq:fixpoint}
  v_k=-\nabla F(x_k),\qquad x_{k+1}=\nabla T_\sigma^*(v_k).
\end{equation}
We assume that $x_0$ and every oracle point lie in $\Omega_{F_\sigma}$. Convexity then keeps every relaxed step in this set.}

The map in \eqref{eq:fixpoint} is the oracle of the generalized conditional-gradient method of Bredies, Lorenz, and Maass \cite{BLM}. It linearizes $F$ and keeps $T_\sigma$ unchanged:
\begin{equation}\label{eq:oracle}
x_k^{++}
  :=\argmin_{x\in\mathbb R^m}
  \{\langle\nabla F(x_k),x\rangle+T_\sigma(x)\}.
\end{equation}
Strong convexity makes the oracle single-valued even when $T_\sigma$ is nonsmooth.

For a convex function $h$ differentiable at $y\in\operatorname{ri}\dom h$, its Bregman divergence is
\begin{equation}\label{eq:Bregman}
  D_h(x,y):=h(x)-h(y)-\langle\nabla h(y),x-y\rangle\geq0.
\end{equation}
Here $x\in\dom h$, and $x-y$ lies in the direction space of $\operatorname{aff}(\dom h)$. For ease of presentation, the primal Bregman formulas below assume that $T_\sigma$ is differentiable on $\Omega_T$, with gradients taken relative to $\operatorname{aff}(\dom T_\sigma)$. The oracle and dual formulations cover nonsmooth $T_\sigma$.

{The Bregman remainder of $F_\sigma$ splits as
$
  D_{F_\sigma}(x,x_k)=D_F(x,x_k)+D_{T_\sigma}(x,x_k).
$
Dropping $D_F$ yields the lower model
\begin{equation}\label{eq:BregT}
x_k^{++}
 =\argmin_{x\in\dom F_\sigma}
  \bigl\{F_\sigma(x_k)+\langle\nabla F_\sigma(x_k),x-x_k\rangle
                  +D_{T_\sigma}(x,x_k)\bigr\}.
\end{equation}
Merging the terms involving $T_\sigma$ gives
\begin{equation}\label{eq:BregT2}
x_k^{++}
 =\argmin_{x\in\dom F_\sigma}
  \bigl\{F(x_k)+\langle\nabla F(x_k),x-x_k\rangle+T_\sigma(x)\bigr\},
\end{equation}
whose minimizer is the oracle point in \eqref{eq:oracle}. The terms involving $\nabla T_\sigma(x_k)$ cancel, so the model also applies to nonsmooth $T_\sigma$.

Thus \eqref{eq:BregT} is the unit-step primal mirror-descent model for $F_\sigma$ with mirror $T_\sigma$. It is a lower model and need not decrease $F_\sigma$. A generalized conditional-gradient method \cite{BLM} with exact line search chooses
$$
\eta_k\in\argmin_{\eta\in[0,1]}F_\sigma\bigl(x_k+\eta(x_k^{++}-x_k)\bigr)
$$
and sets $x_k^+=x_k+\eta_k(x_k^{++}-x_k)$. The superscript $+$ denotes the relaxed oracle step. The superscript $++$ indicates the full oracle step, which lies beyond $x_k^+$ when $0<\eta_k<1$.

For a general step size $\eta_k>0$, primal mirror descent uses the model
\begin{equation}\label{eq:BregTmirror}
  \min_{x\in\dom F_\sigma}
  \bigl\{F_\sigma(x_k)+\langle\nabla F_\sigma(x_k),x-x_k\rangle
                  +\eta_k^{-1}D_{T_\sigma}(x,x_k)\bigr\}.
\end{equation}
When $\eta_k\ne1$, this model requires $\nabla T_\sigma(x_k)$, which may be costly. For nonsmooth $T_\sigma$, its extension uses a chosen subgradient at $x_k$ in both the linear model and the Bregman term; these contributions cancel only when $\eta_k=1$.

We next identify the relaxed step as mirror descent on the dual. Relative smoothness then gives a prescribed step size without exact line search.}

\subsection{Dual formulation}

{The Fenchel--Moreau theorem gives
\begin{equation}\label{eq:dualT}
  T_\sigma(x)=\sup_{v\in\mathbb R^m}
  \{\langle v,x\rangle-T_\sigma^*(v)\}.
\end{equation}
Substitution into \eqref{eq:reg} yields
$$
  \min_x F_\sigma(x)
  =\min_x\sup_v
  \{F(x)+\langle v,x\rangle-T_\sigma^*(v)\}.
$$
The minimizer lies in $\Omega_{F_\sigma}\subset\operatorname{int}\dom F\cap\dom T_\sigma$, where $F$ is continuous. Fenchel duality therefore permits the exchange of the infimum and supremum, and the dual maximum is attained. For fixed $v$,
$$
  \inf_x\{F(x)+\langle v,x\rangle\}
  =-\sup_x\{\langle -v,x\rangle-F(x)\}
  =-F^*(-v).
$$
Consequently,
$$
  \min_x F_\sigma(x)
  =\max_v\{-F^*(-v)-T_\sigma^*(v)\}
  =-\min_v f(v),
$$
where
\begin{equation}\label{eq:dual-objective-fenchel}
  f(v):=F_-^*(v)+T_\sigma^*(v),\qquad
  F_-^*(v):=F^*(-v).
\end{equation}
Set $\Omega_{F^*}:=-\operatorname{int}\dom F^*$. We call $F_-^*$ the reflected conjugate of $F$.}
The function $f$ is the negative Fenchel dual objective; it differs from the conjugate $F_\sigma^*$.

\begin{lemma}[Dual objective]\label{lem:dual-objective-properties}
The function $f$ is Legendre with $\operatorname{int}\dom f=\Omega_{F^*}$. For $v\in\Omega_{F^*}$, set $x(v):=\nabla F^*(-v)$ and $x^{++}(v):=\nabla T_\sigma^*(v)$. Then
$$
\nabla f(v)=-\nabla F^*(-v)+\nabla T_\sigma^*(v)=x^{++}(v)-x(v).
$$
For $u,v\in\Omega_{F^*}$,
$$
D_f(v,u)=D_{F_-^*}(v,u)+D_{T_\sigma^*}(v,u)\geq D_{F_-^*}(v,u).
$$
The unique minimizer is $v^\star=-\nabla F(x^\star)$, with
$$
\nabla f(v^\star)=0,\qquad f(v^\star)=-F_\sigma(x^\star),\qquad D_f(v,v^\star)=f(v)-f(v^\star).
$$
\end{lemma}
\begin{proof}
Adding the finite differentiable convex function $T_\sigma^*$ to the Legendre function $F_-^*$ preserves its domain and strict convexity. At a finite boundary point, $\|\nabla F_-^*\|$ diverges while $\nabla T_\sigma^*$ stays bounded, so $f$ is essentially smooth. The gradient and divergence formulas follow by differentiation and additivity. Equation~\eqref{eq:regEuler} gives $\nabla f(v^\star)=0$; strict convexity gives uniqueness, and Fenchel equality gives the optimal value.
\end{proof}

\subsection{Mirror descent on the dual}

The identities $\nabla F_-^*(v_k)=-x_k$ and $\nabla f(v_k)=x_k^{++}-x_k$ turn the iteration
\begin{equation}\label{eq:relaxed}
x_{k+1}=x_k+\eta_k(x_k^{++}-x_k),\qquad 0<\eta_k\leq1,
\end{equation}
 into mirror descent for $f$ with mirror function $F_-^*$~\cite{Bach,Pena}:
\begin{equation}\label{eq:dualmirror}
\nabla F_-^*(v_{k+1})=\nabla F_-^*(v_k)-\eta_k\nabla f(v_k).
\end{equation}
Equivalently,
\begin{equation}\label{eq:dualMD}
  v_{k+1}=\argmin_{v\in\Omega_{F^*}}
  \bigl\{ f(v_k)+\langle\nabla f(v_k),v-v_k\rangle
                    + \eta_k^{-1}D_{F_-^*}(v,v_k)\bigr\}.
\end{equation}

The decomposition $f=F_-^*+T_\sigma^*$ gives relative strong convexity with modulus one:
\begin{equation}\label{eq:relative-lower}
  D_f(u,v)=D_{F_-^*}(u,v)+D_{T_\sigma^*}(u,v)\geq D_{F_-^*}(u,v),
  \qquad u,v\in\Omega_{F^*}.
\end{equation}
We use relative smoothness in the sense of Lu, Freund, and Nesterov \cite{LFN}.
Let $V\subseteq\Omega_{F^*}$ be a convex region containing $v^\star$ and all iterates. We assume that, for some finite $L>1$,
\begin{equation}\label{eq:L}
D_f(u,v)\leq L D_{F_-^*}(u,v),\qquad u,v\in V.
\end{equation}
By additivity, this condition is equivalent to
\begin{equation}\label{eq:relative-upper}
  D_{T_\sigma^*}(u,v)\leq(L-1)D_{F_-^*}(u,v),
  \qquad u,v\in V.
\end{equation}

Recall $x(v)=\nabla F^*(-v)$. If $F$ and $T_\sigma^*$ are $C^2$ and $\nabla^2F(x(v))\succ0$, then
$
  \nabla^2 F_-^*(v)=\bigl[\nabla^2F(x(v))\bigr]^{-1}.
$
On an open convex region $V$, an admissible bound, when finite, is
\begin{equation}\label{eq:relative-L}
  L = 1+\sup_{v\in V}\lambda_{\max}
  \left(
\nabla^2F(x(v))
    \nabla^2T_\sigma^*(v)
  \right).
\end{equation}
These Hessian assumptions are needed only for this formula, not for the definition \eqref{eq:L}.

The following standard estimate motivates curvature backtracking \cite{LFN}. The proof combines the mirror three-point identity with relative strong convexity to contract the distance to the solution, then uses the tested curvature bound to decrease the dual objective.

\begin{lemma}[Mirror descent with a tested curvature bound]\label{lem:mirror-tested-curvature}
Let $x_{k+1}=x_k-\frac{1}{L_k}g(x_k)$, where $g(x_k)=x_k-x_k^{++}$, $L_k\geq1$, and $v_j=-\nabla F(x_j)$. Suppose the step is admissible and
\begin{equation}\label{eq:mirror-tested-curvature}
D_f(v_{k+1},v_k)\leq L_kD_F(x_k,x_{k+1}).
\end{equation}
Then
\begin{equation}\label{eq:mirror-tested-rate}
D_F(x_{k+1},x^\star)+\frac{1}{L_k}D_f(v_{k+1}, v^\star)\leq\left(1-\frac{1}{L_k}\right)D_F(x_k,x^\star),
\end{equation}
and
\begin{equation}\label{eq:mirror-dual-descent}
f(v_{k+1})-f(v_k)\leq-L_kD_F(x_{k+1},x_k).
\end{equation}
If every step satisfies this test, then
\begin{equation}\label{eq:linear-rate}
D_F(x_{k+1},x^\star) + \frac{1}{L_k}D_f(v_{k+1}, v^\star) \leq\prod_{j<k+1}(1-\frac{1}{L_j})D_F(x_0,x^\star).
\end{equation}
For constant $L_j=L$, the factor in \eqref{eq:linear-rate} is $(1-1/L)^{k+1}$.
\end{lemma}
\begin{proof}
Work directly with the primal Bregman distance. Its three-point identity and $x_{k+1}-x_k=L_k^{-1}\nabla f(v_k)$ give
$$
\begin{aligned}
&D_F(x_{k+1},x^\star)-D_F(x_k,x^\star)\\
&\quad=\langle\nabla F(x_{k+1})-\nabla F(x^\star),x_{k+1}-x_k\rangle-D_F(x_k,x_{k+1})\\
&\quad=\frac1{L_k}\langle\nabla f(v_k),v^\star-v_{k+1}\rangle-D_F(x_k,x_{k+1})\\
&\quad=-\frac1{L_k}\bigl[D_f(v_{k+1},v^\star)+D_f(v^\star,v_k)-D_f(v_{k+1},v_k)\bigr]-D_F(x_k,x_{k+1}).
\end{aligned}
$$
The tested bound \eqref{eq:mirror-tested-curvature} cancels the last Bregman term, and relative strong convexity gives $D_f(v^\star,v_k)\geq D_F(x_k,x^\star)$. This proves \eqref{eq:mirror-tested-rate}.

For the objective decrease, use
$$
\nabla f(v_k)=L_k(x_{k+1}-x_k),\qquad
v_{k+1}-v_k=-(\nabla F(x_{k+1})-\nabla F(x_k)).
$$
The symmetric Bregman identity then gives
$$
\begin{aligned}
f(v_{k+1})-f(v_k)
&=\langle\nabla f(v_k),v_{k+1}-v_k\rangle+D_f(v_{k+1},v_k)\\
&= -L_k\bigl[D_F(x_k,x_{k+1})+D_F(x_{k+1},x_k)\bigr] + D_f(v_{k+1},v_k)\\
&\leq -L_kD_F(x_{k+1},x_k).
\end{aligned}
$$

Drop the nonnegative objective gap in \eqref{eq:mirror-tested-rate}, iterate up to $k$, and substitute the resulting bound for $D_F(x_k,x^\star)$ back into \eqref{eq:mirror-tested-rate}. This proves \eqref{eq:linear-rate}.
\end{proof}

Algorithm~\ref{alg:base-backtracking} uses the measured ratio to enlarge curvature after rejection and initialize the next iteration. Within iteration $k$, $L_k$ denotes the current trial curvature and $L_{k+1}$ stores its measured value,
\begin{equation}\label{eq:monitor-measured-curvature}
L_{k+1}=1+\frac{D_{T_\sigma^*}(v_{k+1},v_k)}{D_F(x_k,x_{k+1})} = \frac{D_{f}(v_{k+1},v_k)}{D_{F_-^*}(v_{k+1},v_k)}.
\end{equation}

To \emph{evaluate} a trial $x_{k+1}$ here means to compute $F(x_{k+1})$, $v_{k+1}=-\nabla F(x_{k+1})$, and $T_\sigma^*(v_{k+1})$. The test uses the cached oracle $x_k^{++}$ through
$$
D_{T_\sigma^*}(v_{k+1},v_k)
=T_\sigma^*(v_{k+1})-T_\sigma^*(v_k)
-\langle x_k^{++},v_{k+1}-v_k\rangle.
$$
The new oracle $x_{k+1}^{++}=\nabla T_\sigma^*(v_{k+1})$ is therefore computed only after the trial passes. This avoids gradient evaluations of $T_\sigma^*$ at rejected points; the test still requires its value.

\begin{algorithm}[H]
\caption{Adaptive GCG}\label{alg:base-backtracking}
\renewcommand{\baselinestretch}{1.15}\selectfont
\begin{algorithmic}[1]
\REQUIRE $x_0\in\Omega_{F_\sigma}$, $L_0\geq1$ (default $1$), tolerance $\varepsilon_{\rm tol}$.
\STATE Evaluate $x_0$ and compute $x_0^{++}=\nabla T_\sigma^*(v_0)$.
\FOR{$k=0,1,\ldots$}
\STATE \textbf{if} $D_F(x_k^{++},x_k)\leq\varepsilon_{\rm tol}$ \textbf{then return} $x_k^{++}$.
\STATE Update $x_{k+1}=x_k+L_k^{-1}(x_k^{++}-x_k)$.
\STATE Evaluate $x_{k+1}$.
\STATE Compute $L_{k+1} = 1+ D_{T_\sigma^*}(v_{k+1},v_k)/D_F(x_k,x_{k+1})$.
\WHILE{$L_{k+1} >L_k$}
\STATE Enlarge $L_k\gets\max\{2L_k,L_{k+1}\}$.
\STATE Update $x_{k+1}=x_k+L_k^{-1}(x_k^{++}-x_k)$.
\STATE Evaluate $x_{k+1}$.
\STATE Compute $L_{k+1} = 1+ D_{T_\sigma^*}(v_{k+1},v_k)/D_F(x_k,x_{k+1})$.
\ENDWHILE
\STATE Compute $x_{k+1}^{++}=\nabla T_\sigma^*(v_{k+1})$.
\ENDFOR
\end{algorithmic}
\end{algorithm}

The stopping test is justified by Proposition~\ref{prop:acc-primal-dual-gap} below.

\begin{corollary}[Linear convergence of adaptive GCG]\label{cor:adaptive-gcg-rate}
Suppose all trials of Algorithm~\ref{alg:base-backtracking} are admissible, \eqref{eq:L} holds on every tested pair, and $1\leq L_0\leq L$. Every search terminates with accepted curvature $L_k\leq2L$, and
$$
D_F(x_k,x^\star)\leq\left(1-\frac1{2L}\right)^kD_F(x_0,x^\star).
$$
\end{corollary}
\begin{proof}
The measured ratio lies in $[1,L]$, so each iteration starts with curvature at most $L$. A rejection implies $L_k<L_{k+1}\leq L$. Thus the update $L_k\gets\max\{2L_k,L_{k+1}\}$ remains below $2L$. The test passes whenever $L_k\geq L$, and each rejection at least doubles $L_k$, so the search terminates. The accepted ratio initializes the next iteration in $[1,L]$, closing the induction. Applying Lemma~\ref{lem:mirror-tested-curvature} with $L_k\leq2L$ proves the rate.
\end{proof}

\subsection{The primal--dual gap}
Define the Fenchel gap of the regularizer by
\begin{equation}\label{eq:primal-dual-gaps}
G(y,v):=T_\sigma(y)+T_\sigma^*(v)-\langle v,y\rangle.
\end{equation}
By definition \eqref{eq:dualT}, the gap is nonnegative. Fenchel equality gives $G(y,v)=0$ precisely when $v\in\partial T_\sigma(y)$, or equivalently $y=\nabla T_\sigma^*(v)$. This relation alone does not imply optimality for \eqref{eq:reg}. In contrast, strong duality gives the Lyapunov function
\begin{equation}\label{eq:gap-objectives}
\mathcal E(y,v):=F_\sigma(y)+f(v)=[F_\sigma(y)-F_\sigma(x^\star)]+[f(v)-f(v^\star)] \geq 0,
\end{equation}
so $\mathcal E$ controls both objective gaps.
\begin{lemma}[Primal--dual gap decomposition]\label{lem:monitor-gap-decomposition}
For admissible $x,y$ and $v=-\nabla F(x)$,
\begin{equation}\label{eq:monitor-gap-decomposition}
\mathcal E(y,v)=G(y,v)+D_F(y,x).
\end{equation}
Under the standing assumptions, $\mathcal E(y,-\nabla F(x))=0$ if and only if $x=y=x^\star$.
\end{lemma}
\begin{proof}
Since $v=-\nabla F(x)$, Fenchel equality gives $F^*(-v)=-F(x)-\langle v,x\rangle$. Expanding the gaps therefore yields
$$
\begin{aligned}
\mathcal E(y,v)-G(y,v)
&=F(y)+F^*(-v)+\langle v,y\rangle\\
&=F(y)-F(x)-\langle\nabla F(x),y-x\rangle=D_F(y,x).
\end{aligned}
$$
Both terms in \eqref{eq:monitor-gap-decomposition} are nonnegative. Strict convexity of $F$ gives $D_F(y,x)=0$ only when $y=x$, while $G(y,v)=0$ gives $v\in\partial T_\sigma(y)$. If both vanish, then $0\in\nabla F(x)+\partial T_\sigma(x)$, so $x=y=x^\star$. Conversely, optimality at $x^\star$ makes both terms zero.
\end{proof}

\begin{proposition}[A computable stopping test]\label{prop:acc-primal-dual-gap}
For admissible $x$, let $v=-\nabla F(x)$ and $x^{++}=\nabla T_\sigma^*(v)$. Then
\begin{equation}\label{eq:acc-primal-dual-gap}
F_\sigma(x^{++})+f(v)=D_F(x^{++},x),\qquad
0\leq F_\sigma(x^{++})-F_\sigma(x^\star)\leq D_F(x^{++},x).
\end{equation}
\end{proposition}
\begin{proof}
Fenchel equality gives $G(x^{++},v)=0$. Apply Lemma~\ref{lem:monitor-gap-decomposition} with $y=x^{++}$ and use \eqref{eq:gap-objectives}.
\end{proof}

The stopping test uses only $F$, $\nabla F$, and the oracle output. In particular, it requires no value of $T_\sigma$ at the returned point.
The quantity $D_F(y_k,x_k)$ alone is not a stopping certificate: it vanishes when $y_k=x_k$, even away from the solution. The missing Fenchel gap $G(y_k,v_k)$ need not vanish, whereas $G(x_k^{++},v_k)=0$ makes the oracle stopping gap exact.

\section{Adaptive acceleration of generalized conditional gradient methods}\label{sec:discrete-convergence}
We discretize a two-state flow, derive a computable Lyapunov certificate, and use it to choose the trial parameters.
Throughout this section, all primal states and oracle points are assumed to lie in $\Omega_{F_\sigma}$. For the algorithms below, this follows from admissible initial states and oracle outputs, since every update is a convex combination.

\subsection{Hessian-driven Nesterov accelerated gradient method}\label{sec:continuous-convergence}
The Hessian-driven Nesterov accelerated gradient method (HNAG)~\cite{CL} combines variable and operator splitting~\cite{ChenHaoWei2025Accelerated} with Hessian-driven damping~\cite{AlvarezAttouchBolteRedont2002,AttouchPeypouquetRedont2016,AttouchChbaniFadiliRiahi2022First-order}. Applying this splitting to the oracle residual $g(x)=x-x^{++}=-\nabla f(v)$, with $v=-\nabla F(x)$ and $x^{++}=\nabla T_\sigma^*(v)$, gives the two-state flow
\begin{equation}\label{eq:acc-flow}
\begin{aligned}
x'&=y-x-\beta g(x),\\
y'&=x-y-g(x) = x^{++}-y.
\end{aligned}
\end{equation}
For constant $\beta>0$ and $C^2$ functions $F$ and $T_\sigma^*$, eliminating $y$ gives
\begin{equation}\label{eq:acc-flow-second-order}
x''+2x'+\beta Dg(x)x'+(1+\beta)g(x)=0.
\end{equation}
Here $Dg(x)=I+\nabla^2T_\sigma^*(v)\nabla^2F(x)$, and the damping term uses $Dg(x)x'=\frac{\mathrm d}{\mathrm dt}g(x)$. Discretizing this derivative through residual differences uses only gradients and oracle values.

A semi-implicit splitting of \eqref{eq:acc-flow} gives
$$
\frac{x_{k+1}-x_k}{\alpha_k}=y_k-x_{k+1}-\beta_k g(x_k),\qquad
\frac{y_{k+1}-y_k}{\alpha_k}=x_{k+1}^{++}-y_{k+1}.
$$
Set $\alpha_k\beta_k=1/L_k$, with $L_k\geq1$ and $\alpha_k>0$. A balanced trial uses $\alpha_k=\beta_k=1/\sqrt{L_k}$. The resulting oracle iteration is
\begin{equation}\label{eq:acc-scheme}
\begin{aligned}
x_{k+1}&=\frac{x_k-\frac{1}{L_k}g(x_k)+\alpha_k y_k}{1+\alpha_k}
=\frac{x_k^++\alpha_k y_k}{1+\alpha_k},
\qquad x_k^+=x_k-\frac{1}{L_k}g(x_k),\\
y_{k+1}&=\frac{y_k+\alpha_k x_{k+1}^{++}}{1+\alpha_k}.
\end{aligned}
\end{equation}
Both updates are convex combinations. One new oracle evaluation at $v_{k+1}=-\nabla F(x_{k+1})$ gives $x_{k+1}^{++}=\nabla T_\sigma^*(v_{k+1})$, the most expensive operation. The states in \eqref{eq:acc-scheme} and the exact residual identities below denote coupled trial values; see Figure~\ref{fig:monitor-iteration}.

\begin{figure}[tbp]
\centering
\includegraphics[width=0.5\textwidth]{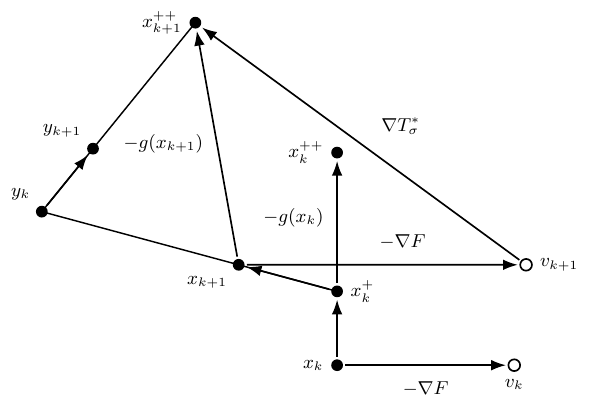}
\caption{Geometry of the coupled HNAG trial.}
\label{fig:monitor-iteration}
\end{figure}

\subsection{Contraction of the duality gap}
Recall that $G(y,v)$ is the Fenchel gap of $T_\sigma$, while $\mathcal E(y,v)$ controls both objective gaps:
$$
G(y,v):=T_\sigma(y)+T_\sigma^*(v)-\langle v,y\rangle,\qquad
\mathcal E(y,v):=F_\sigma(y)+f(v).
$$
The $y$-update contracts $G$ with the new dual variable held fixed.

\begin{lemma}[Fenchel-gap contraction]\label{lem:monitor-fenchel}
Let $\alpha_k\geq0$, $v_{k+1}=-\nabla F(x_{k+1})$, and $x_{k+1}^{++}=\nabla T_\sigma^*(v_{k+1})$. The auxiliary update $y_{k+1}=(y_k+\alpha_k x_{k+1}^{++})/(1+\alpha_k)$ satisfies
\begin{equation}\label{eq:monitor-fenchel-contraction}
G(y_{k+1},v_{k+1})\leq \frac{1}{1+\alpha_k}G(y_k,v_{k+1}).
\end{equation}
\end{lemma}
\begin{proof}
The function $G(\cdot,v_{k+1})$ is convex, and Fenchel equality gives $G(x_{k+1}^{++},v_{k+1})=0$. Applying convexity to the auxiliary update yields
$$
(1+\alpha_k)G(y_{k+1},v_{k+1})
\leq G(y_k,v_{k+1})+\alpha_kG(x_{k+1}^{++},v_{k+1})
=G(y_k,v_{k+1}),
$$
which is \eqref{eq:monitor-fenchel-contraction}.
\end{proof}

This contraction uses only convexity and the $y$-update. The $x$-state enters through the certificate below.
Recall that $\mathcal E(y,v)=G(y,v)+D_F(y,x)$ for $v=-\nabla F(x)$. The certificate below controls the change of the Bregman term together with the dual objective.

\subsection{A computable one-step certificate}\label{sec:monitor-certificate}
Write $v_k=-\nabla F(x_k)$ and $\mathcal E_k:=\mathcal E(y_k,v_k)$, where $\mathcal E(y,v)=F_\sigma(y)+f(v)$ bounds both objective gaps by \eqref{eq:gap-objectives}. Computing $\mathcal E_k$ directly requires $T_\sigma(y_k)$, which may involve an inner optimization solve even when the conjugate oracle is explicit. 

Fix an evaluated point $x_{k+1}$ and write $\Delta f_k:=f(v_{k+1})-f(v_k)$. For $\alpha\geq0$ and $y_{k+1}(\alpha):=(y_k+\alpha x_{k+1}^{++})/(1+\alpha)$, define the residual
\begin{equation}\label{eq:monitor-residual}
\mathcal R_k(y_{k+1}(\alpha)):=\Delta f_k+(1+\alpha)D_F(y_{k+1}(\alpha),x_{k+1})-D_F(y_k,x_{k+1}).
\end{equation}
Here varying $\alpha$ changes only $y_{k+1}(\alpha)$; the evaluated point $x_{k+1}$, its dual variable, and its oracle remain fixed. In a coupled trial of \eqref{eq:acc-scheme}, both states depend on the trial parameter.
\begin{proposition}[Computable one-step certificate]\label{prop:monitor}
For every $\alpha\geq0$, 
\begin{equation}\label{eq:monitor-bound}
(1+\alpha)\mathcal E_{k+1}-\mathcal E_k\leq\mathcal R_k(y_{k+1}(\alpha))\leq \Delta f_k+\alpha D_F(x_{k+1}^{++},x_{k+1}).
\end{equation}
Hence $\mathcal R_k(y_{k+1}(\alpha))\leq0$ implies
\begin{equation}\label{eq:aod-step}
\mathcal E_{k+1}\leq\frac{\mathcal E_k}{1+\alpha}.
\end{equation}
If $\Delta f_k<0<D_F(x_{k+1}^{++},x_{k+1})$, the parameter
\begin{equation}\label{eq:aod-alpha}
\alpha=\frac{-\Delta f_k}{D_F(x_{k+1}^{++},x_{k+1})}
\end{equation}
gives $\mathcal R_k(y_{k+1}(\alpha))\leq0$ and hence \eqref{eq:aod-step} for the recomputed $y_{k+1}$.
\end{proposition}
\begin{proof}
Lemma~\ref{lem:monitor-fenchel} contracts the Fenchel gap
$$
(1+\alpha)G(y_{k+1},v_{k+1})\leq G(y_k,v_{k+1}).
$$
Lemma~\ref{lem:monitor-gap-decomposition} at $x=x_{k+1}$ gives $G(y_{k+1},v_{k+1})=\mathcal E_{k+1}-D_F(y_{k+1},x_{k+1})$ and $G(y_k,v_{k+1})=\mathcal E_k+\Delta f_k-D_F(y_k,x_{k+1})$. Substituting both identities into the preceding inequality gives the first bound in \eqref{eq:monitor-bound}. 

Convexity of $D_F(\cdot,x_{k+1})$ bounds its Bregman contribution:
$$
\begin{aligned}
(1+\alpha)D_F(y_{k+1},x_{k+1})&\leq D_F(y_k,x_{k+1})+\alpha D_F(x_{k+1}^{++},x_{k+1}).
\end{aligned}
$$
This gives the second inequality of \eqref{eq:monitor-bound}. 
\end{proof}

Since $v_k=-\nabla F(x_k)$, Fenchel equality gives
\begin{equation}\label{eq:monitor-dual-value}
f(v_k)=T_\sigma^*(v_k)-F(x_k)-\langle v_k,x_k\rangle.
\end{equation}
Thus the dual value is computed from $F$, $\nabla F$, and the oracle. 
The remaining terms in the residual \eqref{eq:monitor-residual} use $F(y_k)$, $F(y_{k+1})$, $F(x_{k+1})$, and $\nabla F(x_{k+1}) = -v_{k+1}$.

The parameter \eqref{eq:aod-alpha} keeps $x_{k+1}$ and its oracle and changes only $y_{k+1}$. It uses $F(x_{k+1}^{++})$ for the stopping gap, which the next stopping test needs anyway, and one new value $F(y_{k+1})$.

\subsection{The algorithm and the flow map}\label{sec:adaptivity}\label{sec:monitor-algorithm}

Initialize $y_0=x_0$. The subscript $k+1$ denotes the current candidate until it is accepted or discarded. Figure~\ref{fig:flow-map} shows one iteration. To \emph{evaluate} $x_j$ means to compute $F(x_j)$, $v_j=-\nabla F(x_j)$, and $T_\sigma^*(v_j)$. A full oracle call also computes $x_j^{++}=\nabla T_\sigma^*(v_j)$. In the mirror descent branch, \eqref{eq:mirror-dual-descent} gives $\Delta f_k<0$, and the stopping test at the new point ensures that the denominator in \eqref{eq:aod-alpha} is positive.

\emph{Cost.} An accepted trial uses one oracle call and one new value $F(y_{k+1})$ for the certificate; the stopping test also needs $F(x_k^{++})$. A mirror-descent step pays for the discarded trial and each base trial, then evaluates $F(x_{k+1}^{++})$ and $F(y_{k+1})$. Cached values supply all remaining terms in the residual and curvature test.

\begin{algorithm}[H]
\caption{Adaptive HNAG (AHNAG)}\label{alg:monitor}\label{alg:adapt-hnag}\label{alg:aod}
\setlength{\abovedisplayskip}{3pt} \setlength{\belowdisplayskip}{3pt}
\renewcommand{\algorithmiccomment}[1]{\hfill#1}
\begin{algorithmic}[1]
\REQUIRE $x_0\in\Omega_{F_\sigma}$, $L_0\geq1$, $\rho\in(0,1)$ (default $0.98$), $r>1/\rho$ (default $1.25$), $L_{\max}\geq\max\{L_0,2L\}$ for the convergence guarantees, tolerance $\varepsilon_{\rm tol}$.
\STATE Set $y_0=x_0$; evaluate $x_0$ and compute $x_0^{++}$.
\FOR{$k=0,1,\ldots$}
\STATE \textbf{if} $D_F(x_k^{++},x_k)\leq\varepsilon_{\rm tol}$ \textbf{then return} $x_k^{++}$.
\STATE Set $\alpha_k=1/\sqrt{L_k}$ and $x_{k+1}=(x_k-L_k^{-1}g(x_k)+\alpha_k y_k)/(1+\alpha_k)$.
\STATE Evaluate $x_{k+1}$ and compute $x_{k+1}^{++}$.
\STATE Set $y_{k+1}=(y_k+\alpha_kx_{k+1}^{++})/(1+\alpha_k)$; evaluate $\mathcal R_k$.
\IF[(mirror descent)]{$\mathcal R_k>0$}
\STATE Set $x_{k+1}=x_k-L_k^{-1}g(x_k)$ and evaluate it.
\WHILE{$D_f(v_{k+1},v_k)> L_kD_F(x_k,x_{k+1})$}
\STATE Set $L_k\gets2L_k$; recompute $x_{k+1}=x_k-L_k^{-1}g(x_k)$ from $x_k$ and evaluate it.
\ENDWHILE
\STATE Compute $x_{k+1}^{++}$ and $D_F(x_{k+1}^{++},x_{k+1})$.
\STATE \textbf{if} $D_F(x_{k+1}^{++},x_{k+1})\leq\varepsilon_{\rm tol}$ \textbf{then return} $x_{k+1}^{++}$.
\STATE Set $\alpha_k=-\Delta f_k/D_F(x_{k+1}^{++},x_{k+1})$ and $y_{k+1}=(y_k+\alpha_kx_{k+1}^{++})/(1+\alpha_k)$.
\STATE Set $L_{k+1}=\min\{rL_k,L_{\max}\}$.
\ELSE[(accept)]
\STATE Set $L_{k+1}=\max\{1,\rho L_k\}$.
\ENDIF
\STATE Cache the accepted function values and $x_{k+1}^{++}$.
\ENDFOR
\end{algorithmic}
\end{algorithm}

\begin{figure}[htbp]
\centering
\includegraphics[width=0.95\textwidth]{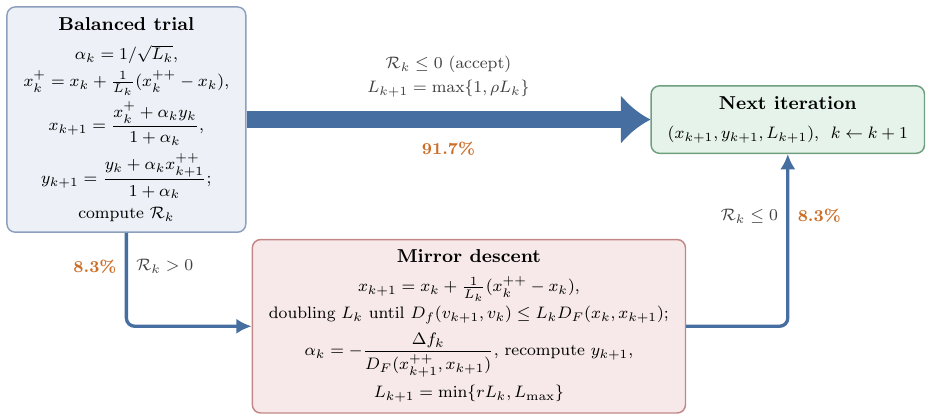}
\caption{One iteration of AHNAG: accept the balanced trial or take a certified mirror-descent step. The percentages are asymptotic bounds on the two branches under Theorem~\ref{thm:acc-local}, not measured frequencies.}
\label{fig:flow-map}
\end{figure}

\section{Convergence analysis}\label{sec:convergence}
We first prove global linear convergence at the scale of the base method, then establish a local accelerated rate by counting accepted balanced trials.

\subsection{Proof outline}

An accepted balanced trial contracts the energy with $\alpha_k=1/\sqrt{L_k}$. If the trial fails, the mirror-descent fallback terminates by the relative-smoothness test of Lemma~\ref{lem:mirror-tested-curvature} and gives $\Delta f_k<0$. Recomputing $y_{k+1}$ with \eqref{eq:aod-alpha} makes the residual nonpositive. Thus every completed iteration satisfies $\mathcal E_{k+1}\leq(1+\alpha_k)^{-1}\mathcal E_k$.
It remains to bound $\alpha_k$ uniformly away from zero. Lemma~\ref{lem:aod-descent-alpha} gives a lower bound at the scale of the mirror step. Thus every iteration contracts the energy, giving a global linear rate with nonaccelerated dependence on curvature.

Increasing $L_k$ alone does not guarantee that a balanced trial passes: the residual can remain positive as $L_k\to\infty$ (Example~\ref{ex:balanced-failure}). Near the minimizer, however, strict local relative smoothness gives a threshold $L_{\rm acc}>L_\star$ above which every balanced trial passes. Global convergence brings the trajectory into this neighborhood. There, a base step occurs only below $L_{\rm acc}$ and raises the curvature by at least $r>1$, while an accepted trial retains at least a fraction $0<\rho<1$ of the current curvature. Counting these changes gives a positive fraction of balanced steps and the local accelerated rate of Theorem~\ref{thm:acc-local}.

\subsection{Global linear convergence}\label{sec:cag-global}
To obtain uniform constants, we bound the primal candidates on a fixed compact set. Let
$$
\begin{aligned}
K_0&:=\{v:\ f(v)-f(v^\star)\leq\mathcal E_0\},\\
X_0 &:=\{\nabla F^*(-v):\ v\in K_0\},\ X_0^{++}:= \nabla T_\sigma^*(-\nabla F(X_0)),\\
Y_0&:=\{y:\ F_\sigma(y)-F_\sigma(x^\star)\leq\mathcal E_0\},\\
\mathcal X &:= \text{convex hull of } X_0\cup X_0^{++}\cup Y_0.
\end{aligned}
$$
By \eqref{eq:gap-objectives}, a state with $\mathcal E_k\leq\mathcal E_0$ has $v_k\in K_0$, $x_k\in X_0$, and $y_k\in Y_0$. 

\medskip
\noindent
\emph{Assumption {\rm(A)}.} The set $K_0$ is a compact subset of $\Omega_{F^*}$, and $\mathcal X$ is a compact subset of $\Omega_{F_\sigma}$. Fix $H\succ0$ and write $\|z\|_H^2=z^\top Hz$. The function $F$ is $\mu$-strongly convex in $\|\cdot\|_H$, with $\|\nabla F(x)-\nabla F(y)\|_{H^{-1}}\leq\Lambda\|x-y\|_H$ on a convex neighborhood of $\mathcal X$, for some $0<\mu\leq\Lambda$. The relative-smoothness bound \eqref{eq:L} holds with a constant $L$ on a convex region $V\supseteq-\nabla F(\mathcal X)$. 
\medskip

In this subsection, write $F\in\mathcal S_{\mu,\Lambda}$ for these bounds in the fixed metric $H$. A suitable choice of $H$ can make $\mu$ and $\Lambda$ close to one. Then
$$
\frac\mu2\|x-y\|_H^2\leq D_F(x,y)\leq\frac\Lambda2\|x-y\|_H^2, \quad \forall x, y\in \mathcal X.
$$
Strong convexity of $T_\sigma$ makes $\nabla T_\sigma^*$ Lipschitz, so $g$ is Lipschitz; let $\ell_g$ denote its Lipschitz constant in $\|\cdot\|_H$. If $F\in\mathcal S_{\mu,\Lambda}$ on $\mathbb R^m$, the $1/\Lambda$-strong convexity of $F_-^*$ in $\|\cdot\|_{H^{-1}}$ and Lipschitz continuity of $\nabla T_\sigma^*$ also imply a finite global bound in \eqref{eq:L}.

Coercivity of $f$ makes $K_0$ compact as a sublevel set in $\mathbb R^m$, but its containment in $\Omega_{F^*}$ must also be checked. The objective $F_\sigma$ is strongly convex, so its sublevel $Y_0$ is compact. Continuity of $\nabla F^*$ and $\nabla T_\sigma^*$ then makes $X_0$, $X_0^{++}$, and their convex hull with $Y_0$ compact. Admissibility of this hull remains a separate condition: bounded sublevels may meet a domain boundary. A sufficient setting is $F\in\mathcal S_{\mu,\Lambda}$ on $\mathbb R^m$ and $\dom T_\sigma=\mathbb R^m$. Then $F_-^*$, hence $f$, is strongly convex with full domain, all these sets are compact and admissible, and the global relative-smoothness bound follows from the Lipschitz conjugate oracle.

The cap bounds curvature carried between iterations. Under $L_{\max}\geq\max\{L_0,2L\}$, it also bounds every base-search curvature: a rejected base trial has $L_k<L$, so doubling gives $L_k<2L\leq L_{\max}$. The algorithm does not require $L$.

\begin{lemma}[One iteration]\label{lem:aod-descent-alpha}
Suppose $F\in\mathcal S_{\mu,\Lambda}$ on the relevant primal segments, and \eqref{eq:L} holds for all tested dual pairs. Assume $L_{\max}\geq2L$ and start iteration $k$ of Algorithm~\ref{alg:aod} with $L_k\leq L_{\max}$. Then the iteration makes finitely many oracle calls, keeps $L_k\leq L_{\max}$, and either passes a stopping test or ends with $L_{k+1}\leq L_{\max}$ and \eqref{eq:aod-step}, where
\begin{equation}\label{eq:aod-descent-alpha}
\alpha_k\geq\frac{c_{\rm d}}{L_{\max}},\qquad
c_{\rm d}:=\frac{\mu}{\Lambda(1+\ell_g)^2}\leq1.
\end{equation}
\end{lemma}
\begin{proof}
Follow the two branches in Figure~\ref{fig:flow-map}. If the trial passes, Proposition~\ref{prop:monitor} gives \eqref{eq:aod-step} with $\alpha_k=1/\sqrt{L_k}\geq c_{\rm d}/L_{\max}$, so the iteration ends with the accelerated contraction.

Otherwise, the mirror-descent fallback terminates once $L_k\geq L$. Its doubling search keeps its final curvature at most $L_{\max}$, and $L_{k+1}=\min\{rL_k,L_{\max}\}$ gives $L_{k+1}\leq L_{\max}$. Unless a stopping test returns, Lemma~\ref{lem:mirror-tested-curvature} gives $\Delta f_k<0$ and the new stopping gap is positive. Thus the fallback parameter \eqref{eq:aod-alpha} certifies \eqref{eq:aod-step}.

It remains to bound this parameter from below. Since $x_{k+1}-x_k=-L_k^{-1}g(x_k)$, Lipschitz continuity of $g$ gives $\|g(x_{k+1})\|_H\leq(1+\ell_g)\|g(x_k)\|_H$. The dual decrease \eqref{eq:mirror-dual-descent} and the quadratic Bregman bounds therefore yield
$$
\begin{aligned}
\alpha_k
&=\frac{-\Delta f_k}{D_F(x_{k+1}^{++},x_{k+1})}
\geq\frac{L_kD_F(x_{k+1},x_k)}{D_F(x_{k+1}^{++},x_{k+1})} \geq\frac{\mu}{\Lambda L_k}
\frac{\|g(x_k)\|_H^2}{\|g(x_{k+1})\|_H^2}
\geq\frac{c_{\rm d}}{L_k}
\geq\frac{c_{\rm d}}{L_{\max}}.
\end{aligned}
$$
\end{proof}

\begin{theorem}[Global linear convergence]\label{thm:aod-global}
Assume {\rm(A)} and $L_{\max}\geq\max\{L_0,2L\}$. Then every iteration of Algorithm~\ref{alg:aod} terminates, and all accepted states satisfy $L_k\leq L_{\max}$ and
\begin{equation}\label{eq:aod-global-rate}
\mathcal E_k\leq\Bigl(1+\frac{c_{\rm d}}{L_{\max}}\Bigr)^{-k}\mathcal E_0.
\end{equation}
If the algorithm continues indefinitely, then $x_k,y_k\to x^\star$ and the stopping gap $D_F(x_k^{++},x_k)$ tends to zero. Thus every positive stopping tolerance is reached in finitely many iterations.
\end{theorem}
\begin{proof}
For a state with $\mathcal E_k\leq\mathcal E_0$, every primal candidate lies in $\mathcal X$ by convexity, so {\rm(A)} guarantees termination of the base search. Its dual decrease puts the accepted base point in $X_0$ and its oracle in $X_0^{++}$. Thus {\rm(A)} supplies the hypotheses of Lemma~\ref{lem:aod-descent-alpha}, including at the initial state. Its certified decrease keeps $\mathcal E_{k+1}\leq\mathcal E_0$, and its curvature bound permits the next iteration, so induction gives finite termination and \eqref{eq:aod-step} at every completed update. The uniform lower bound \eqref{eq:aod-descent-alpha} then gives \eqref{eq:aod-global-rate} by multiplication.

The convergence is in norm. Strong convexity of $F_\sigma$ and Legendre conjugacy give
$$
\begin{aligned}
F_\sigma(y_k)-F_\sigma(x^\star)&\geq\frac\mu2\|y_k-x^\star\|_H^2,\\
f(v_k)-f(v^\star)&\geq D_{F_-^*}(v_k,v^\star)
=D_F(x^\star,x_k)\geq\frac\mu2\|x_k-x^\star\|_H^2.
\end{aligned}
$$
Thus \eqref{eq:gap-objectives} and \eqref{eq:aod-global-rate} imply
$$
\|x_k-x^\star\|_H^2+\|y_k-x^\star\|_H^2
\leq\frac{2\mathcal E_0}{\mu}
\Bigl(1+\frac{c_{\rm d}}{L_{\max}}\Bigr)^{-k}.
$$
Continuity of the oracle then gives $x_k^{++}\to x^\star$, so $D_F(x_k^{++},x_k)\to0$.
\end{proof}

\begin{remark}\rm
The constant in the global linear rate can be improved under additional metric co-coercivity and Hessian-comparison assumptions, without changing the algorithm. This refinement retains the nonaccelerated dependence on curvature.
\end{remark}

\subsection{Exact residual identities}\label{sec:residual-metric}
The certificate $\mathcal R_k(y_{k+1}(\alpha))$ uses only the $y$-update. The global rate also uses the quantitative decrease of the mirror-descent step to bound $\alpha_k$ below. To explain acceleration, we show how the coupled $x$-update enters the residual $\mathcal R_k$.

\begin{lemma}[Residual at a fixed primal point]\label{lem:monitor-fixed-residual}
Fix an admissible $x_{k+1}$, let $v_j=-\nabla F(x_j)$, and set
$y_{k+1}=(y_k+\alpha_kx_{k+1}^{++})/(1+\alpha_k)$ with $\alpha_k\geq0$.
Then the residual \eqref{eq:monitor-residual} satisfies
\begin{equation}\label{eq:monitor-fixed-residual}
\begin{aligned}
\mathcal R_k={}&-D_f(v_k,v_{k+1})-D_F(y_k,y_{k+1})
-\alpha_kD_F(x_{k+1},y_{k+1})\\
&+\langle g(x_{k+1}),\nabla F(x_{k+1})-\nabla F(x_k)\rangle\\
&+\alpha_k\langle g(x_{k+1}),\nabla F(x_{k+1})-\nabla F(y_{k+1})\rangle.
\end{aligned}
\end{equation}
This identity does not require an update rule for $x_{k+1}$.
\end{lemma}
\begin{proof}
Expanding $D_F(\cdot,x_{k+1})$ at $y_{k+1}$ gives
$$
\begin{aligned}
&D_F(y_{k+1},x_{k+1})-D_F(y_k,x_{k+1})\\
&\quad=\langle\nabla F(y_{k+1})-\nabla F(x_{k+1}),y_{k+1}-y_k\rangle-D_F(y_k,y_{k+1}).
\end{aligned}
$$
The auxiliary update gives
\begin{equation}\label{eq:monitor-y-displacement}
\begin{aligned}
y_{k+1}-y_k
&=\alpha_k[x_{k+1}-y_{k+1}-g(x_{k+1})].
\end{aligned}
\end{equation}
Use the symmetric Bregman identity
$$
\langle\nabla F(y_{k+1})-\nabla F(x_{k+1}),y_{k+1} - x_{k+1}\rangle
= D_F(x_{k+1},y_{k+1}) + D_F(y_{k+1},x_{k+1}).
$$

Combining these identities yields
$$
\begin{aligned}
&(1+\alpha_k)D_F(y_{k+1},x_{k+1})-D_F(y_k,x_{k+1})\\
&\quad=-\alpha_kD_F(x_{k+1},y_{k+1})-D_F(y_k,y_{k+1})+\alpha_k\langle g(x_{k+1}),\nabla F(x_{k+1})-\nabla F(y_{k+1})\rangle.
\end{aligned}
$$
Expanding the dual change at $v_{k+1}$ gives
$$
\Delta f_k=\langle g(x_{k+1}),\nabla F(x_{k+1})-\nabla F(x_k)\rangle-D_f(v_k,v_{k+1}).
$$
Adding these two expressions gives \eqref{eq:monitor-fixed-residual}.
\end{proof}

The Bregman terms in \eqref{eq:monitor-fixed-residual} are nonpositive. To control the cross terms for a coupled trial, we now introduce the secant metrics and use the $x$-update. For $H\succ0$, write $\|u\|_H^2=u^\top Hu$ and $\langle u,w\rangle_H=u^\top Hw$; $A\preceq B$ means that $B-A$ is positive semidefinite.

\begin{lemma}[Residual in secant metrics]\label{lem:monitor-secant}
Suppose $F$ is $C^2$ with positive definite Hessian on the relevant segments. Let $H_{xx}$ connect successive $x$-states and $H_{xy}$ connect the new $x$- and $y$-states:
\begin{equation}\label{eq:monitor-secant-metric}
\begin{aligned}
H_{xx}&=\int_0^1\nabla^2F\bigl(x_k+t(x_{k+1}-x_k)\bigr)\,dt,\\
H_{xy}&=\int_0^1\nabla^2F\bigl(y_{k+1}+t(x_{k+1}-y_{k+1})\bigr)\,dt.
\end{aligned}
\end{equation}
For every coupled trial of \eqref{eq:acc-scheme},
\begin{equation}\label{eq:monitor-grouped-residual}
\begin{aligned}
\mathcal R_k={}&-\left[D_f(v_k,v_{k+1})-\frac1{2L_k}\| \nabla f(v_{k+1})-\nabla f(v_k)\|_{H_{xx}}^2\right] -\frac1{2L_k}\|g(x_k)\|_{H_{xx}}^2\\
&+\Bigl[\alpha_k\langle g(x_{k+1}),y_k-y_{k+1}\rangle_{H_{xx}} -D_F(y_k,y_{k+1})-\frac1{2L_k}\|g(x_{k+1})\|_{H_{xx}}^2\Bigr]\\
&+\alpha_k\Bigl[\langle g(x_{k+1}),(H_{xy}-H_{xx})(x_{k+1}-y_{k+1})\rangle -D_F(x_{k+1},y_{k+1})\Bigr].
\end{aligned}
\end{equation}
\end{lemma}
\begin{proof}
Substitute the secant relations
$$
\begin{aligned}
\nabla F(x_{k+1})-\nabla F(x_k)&=H_{xx}(x_{k+1}-x_k),\\
\nabla F(x_{k+1})-\nabla F(y_{k+1})&=H_{xy}(x_{k+1}-y_{k+1})
\end{aligned}
$$
into \eqref{eq:monitor-fixed-residual}, and use the coupled $x$-update in the form
$$
x_{k+1}-x_k=-\frac1{L_k}g(x_k)+\alpha_k(y_k-y_{k+1})+\alpha_k(y_{k+1}-x_{k+1}).
$$
The pairings with $y_{k+1}-x_{k+1}$ cancel except for the metric difference:
$$
\begin{aligned}
\mathcal R_k={}&-\frac1{L_k}\langle g(x_{k+1}),g(x_k)\rangle_{H_{xx}}-D_f(v_k,v_{k+1})\\
&+\alpha_k\langle g(x_{k+1}),y_k-y_{k+1}\rangle_{H_{xx}}-D_F(y_k,y_{k+1})\\
&+\alpha_k\bigl[\langle g(x_{k+1}),(H_{xy}-H_{xx})(x_{k+1}-y_{k+1})\rangle
-D_F(x_{k+1},y_{k+1})\bigr].
\end{aligned}
$$
Polarizing the first pairing and using $g(x) = - \nabla f(v)$ gives \eqref{eq:monitor-grouped-residual}.
\end{proof}
These metrics are used only in the analysis; the algorithm does not evaluate them. 

\begin{example}[Small balanced steps can fail]\label{ex:balanced-failure}
Let $F(x)=e^x$, $T_\sigma(x)=\frac12(x-3)^2$, and $(x_k,y_k)=(0,1)$. Then $x^{++}=3-e^x$, and the coupled trial with $\alpha_k=\alpha=L_k^{-1/2}$ has
$$
x_{k+1}=\frac{\alpha+2\alpha^2}{1+\alpha},\qquad
\mathcal R_k(y_{k+1}(\alpha))=(2e-5)\alpha+O(\alpha^2).
$$
Since $2e-5>0$, every sufficiently small balanced step fails the residual test. Thus, with $\alpha_k=1/\sqrt{L_k}$, increasing $L_k$ alone need not guarantee acceptance: every sufficiently large curvature fails at this state.
\end{example}

A possible modification separates the curvature search from a search for $\alpha_k=\theta/\sqrt{L_k}$ with $\theta \in (0,1)$. However, changing either parameter changes $x_{k+1}$ in \eqref{eq:acc-scheme}, hence its oracle and secant metrics. A curvature test at one value of $\alpha_k$ cannot be retained automatically when $\alpha_k$ changes. We therefore retain a single balanced trial followed, if needed, by the certified mirror step.

\subsection{Local accelerated rate}
Assume $F$ is $C^2$ near $x^\star$ with positive definite Hessian and $T_\sigma^*$ is $C^2$ near $v^\star$. Define the local metric and relative curvature by
\begin{equation}\label{eq:local-relative-curvature}
\begin{aligned}
H_\star&:=\nabla^2F(x^\star),\\
L_\star&:=\lambda_{\max}\bigl(H_\star\nabla^2f(v^\star)\bigr)
=1+\lambda_{\max}\bigl(H_\star\nabla^2T_\sigma^*(v^\star)\bigr).
\end{aligned}
\end{equation}
Thus $L_\star\geq1$ and $\nabla^2f(v^\star)\preceq L_\star H_\star^{-1}$. We first show that balanced trials pass locally above a threshold $L_{\rm acc}>L_\star$. We then count the curvature increases and decreases to obtain a positive fraction of balanced steps and an accelerated rate.

If $F$ and $T_\sigma$ are quadratic, then $H_{xx}=H_{xy}=H_\star$ and there is no metric mismatch. The quadratic dual satisfies
$$
D_f(v_k,v_{k+1})\geq\frac1{2L_\star} \|\nabla f(v_{k+1})-\nabla f(v_k)\|_{H_\star}^2.
$$
Young's inequality in \eqref{eq:monitor-grouped-residual}, with $\alpha_k^2=1/L_k$, therefore gives
$$
\begin{aligned}
\mathcal R_k\leq{}&-\frac1{2L_k}\|g(x_k)\|_{H_\star}^2 -\left(\frac1{2L_\star}-\frac1{2L_k}\right)
\|g(x_{k+1})-g(x_k)\|_{H_\star}^2.
\end{aligned}
$$
Thus every balanced trial with $L_k\geq L_\star$ passes, globally. For nonquadratic functions, we require a strict margin and a neighborhood of the solution.

\begin{proposition}[Local acceptance of balanced trials]\label{lem:balanced-acceptance}\label{prop:monitor-local}
Fix $L_\star<L_{\rm acc}\leq L_{\max}$. There is a neighborhood of $(x^\star,x^\star)$ in which every balanced trial with $L_k\in[L_{\rm acc},L_{\max}]$  satisfies $\mathcal R_k\leq0$, strictly unless $(x_k,y_k)=(x^\star,x^\star)$.
\end{proposition}
\begin{proof}
As the state approaches the solution, the coupled trial and its oracle also converge, uniformly on the curvature interval. Hence $H_{xx},H_{xy}\to H_\star$. Taylor expansion of $f$ at $v^\star$, using $\nabla f(v_j)=-g(x_j)$ and \eqref{eq:local-relative-curvature}, gives locally
\begin{equation}\label{eq:balanced-cocoercivity}
D_f(v_k,v_{k+1})\geq\frac1{L_\star+L_{\rm acc}}
\| \nabla f(v_{k+1})-\nabla f(v_k) \|_{H_{xx}}^2.
\end{equation}
Indeed, the leading quadratic forms are $\frac12\nabla^2f(v^\star)$ and $\nabla^2f(v^\star)H_\star\nabla^2f(v^\star)$, respectively; the latter is bounded by $L_\star\nabla^2f(v^\star)$, and $L_\star+L_{\rm acc}>2L_\star$ leaves a strict margin.

Choose $\epsilon>0$ small enough as specified below. By shrinking the neighborhood, Hessian continuity gives
$$
\begin{aligned}
D_F(y_k,y_{k+1})&\geq\frac{1-\epsilon}{2}\|y_k-y_{k+1}\|_{H_{xx}}^2,\\
D_F(x_{k+1},y_{k+1})&\geq\frac{1-\epsilon}{2}\|x_{k+1}-y_{k+1}\|_{H_{xx}}^2,\\
\|H_{xx}^{-1/2}(H_{xy}-H_{xx})H_{xx}^{-1/2}\|_{\rm op}&\leq\epsilon.
\end{aligned}
$$
Apply Young's inequality to the last two brackets of \eqref{eq:monitor-grouped-residual} and recall that $\alpha_k^2=1/L_k$. With \eqref{eq:balanced-cocoercivity}, this yields
\begin{equation}\label{eq:balanced-residual-bound}
\begin{aligned}
\mathcal R_k\leq{}&-\frac1{2L_k}\|g(x_k)\|_{H_{xx}}^2\\
&-\left(\frac1{L_\star+L_{\rm acc}}-\frac1{2L_k}\right)
\|g(x_{k+1})-g(x_k)\|_{H_{xx}}^2\\
&+\frac{\epsilon+\epsilon^2\sqrt{L_k}}{2L_k(1-\epsilon)}\|g(x_{k+1})\|_{H_{xx}}^2.
\end{aligned}
\end{equation}
To absorb the positive term uniformly for $L_k\in[L_{\rm acc},L_{\max}]$, use
$$
\|g(x_{k+1})\|_{H_{xx}}^2\leq2\|g(x_k)\|_{H_{xx}}^2
+2\|g(x_{k+1})-g(x_k)\|_{H_{xx}}^2,
$$
and choose
$$
0<\epsilon\leq\min\left\{\frac1{\sqrt{L_{\max}}},
\frac{L_{\rm acc}-L_\star}{16(L_{\rm acc}+L_\star)}\right\}.
$$
Indeed, uniformly on the curvature interval,
$$
\frac{\epsilon+\epsilon^2\sqrt{L_k}}{1-\epsilon}
\leq4\epsilon\leq\frac{L_{\rm acc}-L_\star}{4(L_{\rm acc}+L_\star)},
$$
so the last term is absorbed with a strict margin. The result is strictly negative unless $g(x_k)=g(x_{k+1})=0$. Then $x_k=x_{k+1}=x^\star$, and the coupled $x$-update forces $y_k=x^\star$.
\end{proof}

The threshold $L_{\rm acc}$ is used only in the proof; the algorithm does not need to know it or $L_\star$. A larger threshold permits a larger Hessian tolerance $\epsilon$ in the acceptance proof, up to $1/\sqrt{L_{\max}}$, but gives a weaker accelerated-rate constant.

Once the states enter the acceptance neighborhood, a failed balanced trial raises curvature by at least $r$, while an accepted trial reduces it by $\rho$, down to the floor $1$. The cap already bounds the number of increases. To obtain accelerated dependence on $L_\star$, we first sharpen the curvature bound to $2rL_{\rm acc}$ and then count these changes. The sufficient cap condition below ensures that failed trials retain the increase by $r$.

\begin{theorem}[Local accelerated convergence]\label{thm:acc-local}
Assume {\rm(A)} and the local smoothness conditions above, with $L_{\max}\geq\max\{L_0,2L,2rL_\star\}$. Fix $L_\star<L_{\rm acc}\leq2L_\star$. Along any infinite trajectory of Algorithm~\ref{alg:aod}, there is an index $N$ such that $L_k\leq2rL_{\rm acc}$ for every $k\geq N$. If $A_k$ counts accepted balanced trials among iterations $N,\ldots,k-1$, then, for $k\geq N$,
\begin{equation}\label{eq:md-count}
A_k\geq\max\left\{0,
\frac{(k-N)\log r-\log(2r L_{\rm acc})}{\log(r/\rho)}\right\},
\end{equation}
and
\begin{equation}\label{eq:acc-local-energy-rate}
\mathcal E_k\leq\Bigl(1+\frac1{\sqrt{2r L_{\rm acc}}}\Bigr)^{-A_k}\mathcal E_N.
\end{equation}
\end{theorem}
\begin{proof}
Here $L_k$ denotes the curvature at the start of iteration $k$. By global convergence, the states eventually stay in the neighborhood of Proposition~\ref{prop:monitor-local}, where balanced trials pass whenever $L_k\geq L_{\rm acc}$. The strict inequality $L_{\rm acc}>L_\star$ and Hessian continuity also give
\begin{equation}\label{eq:Lacc}
\nabla^2f(v)\prec L_{\rm acc}\nabla^2F_-^*(v)
\qquad\text{for }v\text{ sufficiently close to }v^\star.
\end{equation}
Integrating this inequality on a small convex dual neighborhood gives relative smoothness with constant $L_{\rm acc}$. Since $g(x_k)\to0$, the mirror candidates $x_k-L_k^{-1}g(x_k)$ approach $x^\star$ uniformly for all search curvatures $L_k\geq1$. Thus, eventually, a mirror test passes whenever its search curvature is at least $L_{\rm acc}$.

If $L_k>2rL_{\rm acc}$, the balanced trial passes and $L_{k+1}=\rho L_k$ until the curvature falls below $2rL_{\rm acc}$. By increasing the index $N$ if necessary, we may therefore assume that the states stay in the local neighborhood and $L_N\leq2rL_{\rm acc}$. We now show that both branches preserve this bound.

If the balanced trial passes, then
$$
L_{k+1}=\max\{1,\rho L_k\}\leq2rL_{\rm acc},
\qquad L_{k+1}\geq\rho L_k.
$$
The curvature may reach the floor $1$; the second inequality still holds. If the balanced trial fails, then $L_k<L_{\rm acc}$. Its mirror search ends below $2L_{\rm acc}$: either the first test passes, or the last doubling starts below $L_{\rm acc}$. Since $L_{\max}\geq2rL_\star\geq rL_{\rm acc}>rL_k$, the subsequent update, even if clipped, satisfies
$$
rL_k\leq L_{k+1}<2rL_{\rm acc}.
$$
Thus $L_k\leq2rL_{\rm acc}$ for every $k\geq N$.

Among the $k-N$ iterations starting at $N$, exactly $A_k$ accept balanced trials. Multiplying the lower bounds from the two branches gives
$$
2rL_{\rm acc}\geq L_k\geq
r^{\,k-N-A_k}\rho^{A_k}L_N\geq r^{\,k-N-A_k}\rho^{A_k}.
$$
Taking logarithms and using $A_k\geq0$ proves \eqref{eq:md-count}. Every accepted balanced trial has $\alpha_k=L_k^{-1/2}\geq(2rL_{\rm acc})^{-1/2}$, while mirror steps also decrease the energy. Multiplying the one-step certificates \eqref{eq:aod-step} gives \eqref{eq:acc-local-energy-rate}.
\end{proof}

In particular,
$$
\liminf_{k\to\infty}\frac{A_k}{k-N}\geq\frac{\log r}{\log(r/\rho)}.
$$
For the defaults $r=1.25$ and $\rho=0.98$, this fraction is approximately $0.917$: asymptotically, at least $91.7\%$ of iterations are balanced steps and at most $8.3\%$ use the base branch. The finite-iteration bound \eqref{eq:md-count} includes the initial counting correction.

For fixed $r$ and $\rho$, \eqref{eq:acc-local-energy-rate} gives $O\bigl(\sqrt{L_\star}\log(1/\varepsilon)\bigr)$ iterations as $\varepsilon\downarrow0$, after the finite initial phase. Both objective gaps share this rate. The stopping gap does too: local smoothness of $F$ and Lipschitz continuity of the oracle give $D_F(x_k^{++},x_k)=O(\|x_k-x^\star\|^2)=O(\mathcal E_k)$, using the norm estimate in the global convergence proof.

\section{Numerical experiments}\label{sec:numerics}

We compare AHNAG with adaptive GCG to measure the gain from acceleration beyond curvature adaptation. GCG and HNAG use fixed $L=L_\star$ as reference methods. Adaptive GCG uses Algorithm~\ref{alg:base-backtracking} with $L_0=1$; AHNAG uses Algorithm~\ref{alg:aod} with $L_0=2$, $\rho=0.98$, and $r=1.25$. All methods share the same initial iterate. Adaptive GCG restarts from the ratio in \eqref{eq:monitor-measured-curvature}, whereas AHNAG uses $\rho$ and $r$ to track the curvature at which balanced trials pass.

A \emph{balanced trial} uses $\alpha_k=1/\sqrt{L_k}$; the \emph{base branch} is the mirror-descent fallback. One \emph{oracle call} evaluates both $T_\sigma^*(v)$ and $\nabla T_\sigma^*(v)$. This is the dominant cost: on a one-dimensional grid with $m$ cells, each $T_\sigma^*(v)$ and $\nabla T_\sigma^*(v)$ costs $O(m^2)$, whereas evaluating $F$ and $\nabla F$ costs $O(m\log m)$ through an FFT-based Toeplitz product for granular media or a discrete cosine transform for PNP. In these experiments, the value and gradient require two dense log-domain reductions per species, so a value-only evaluation counts as half a call. Let $n_{\rm b}$ denote the mirror trials. An AHNAG fallback with $n_{\rm b}\geq1$ mirror trials costs $2+(n_{\rm b}-1)/2$ calls. Adaptive GCG costs $1+(n_{\rm b}-1)/2$ calls per iteration. We report total calls, with the cost of rejected trials in parentheses, and round these counts up to integers. The additional evaluations of $F(y_{k+1})$ and the stopping gap require only values of $F$.

NR means that the stopping tolerance was not reached within the iteration cap. CPU times are from single-threaded runs on an Apple M4 processor, timed in one session: the best of three runs in one dimension and one run in three dimensions. Oracle calls are the primary comparison, since times below one second are sensitive to timing noise.

\subsection{Problem setting}
{The JKO method advances Wasserstein gradient flows through successive minimization problems \cite{JKO1998}. Efficient minimization is central to its use for the granular-media and PNP models below.} We solve entropically regularized JKO steps \cite{Peyre,Carlier}:
\begin{equation}\label{eq:num-jko}
  \rho^{k+1}=\argmin_\rho\ \bigl\{\mathrm{OT}_\varepsilon(\rho^k,\rho)/\tau+F(\rho)\bigr\}.
\end{equation}
{Here $\tau>0$ is the JKO {step parameter}, and $\mathrm{OT}_\varepsilon$ is the discrete transport cost with entropy regularization of strength $\varepsilon>0$:
$$
\mathrm{OT}_\varepsilon(a,b)=
\min_{\substack{P\in\mathbb R_+^{m\times m}\\P\one=a,\ P^\top\one=b}}
\sum_{i=1}^m\sum_{j=1}^m\bigl[C_{ij}P_{ij}+\varepsilon P_{ij}(\log P_{ij}-1)\bigr],
$$
where $a,b\in\mathbb R_+^m$ have equal total mass and $C\in\mathbb R^{m\times m}$ is the ground-cost matrix{, $C_{ij}=|x_i-x_j|^2$ on the grid points,} and $0\log0=0$.} {Since the JKO scheme \cite{JKO1998} weights $W_2^2$ by $1/(2h)$ for the time step $h$, \eqref{eq:num-jko} with $\varepsilon\to0$ advances the gradient flow by $h=\tau/2$; we report the parameter $\tau$ throughout.}
{For one JKO step, the regularizer is
$$
T_\sigma(\rho)=\mathrm{OT}_\varepsilon(\rho^k,\rho)/\tau,
$$
with $\rho^k$ fixed. Here $\rho$ denotes the vector of cell masses.

Maximizing over transport plans with the fixed row marginal gives, up to a constant independent of $v$,
$$
\begin{aligned}
T_\sigma^*(v)&=\frac{\varepsilon}{\tau}\sum_{i=1}^m\rho_i^k\log\left(\sum_{j=1}^m e^{(\tau v_j-C_{ij})/\varepsilon}\right),\\
\nabla T_\sigma^*(v)&=\sum_{i=1}^m\rho_i^k\operatorname{softmax}\left(\frac{\tau v-C_{i,:}}{\varepsilon}\right),
\qquad \operatorname{softmax}(z)_j=\frac{e^{z_j}}{\sum_l e^{z_l}}.
\end{aligned}
$$
Here $C_{i,:}$ is the $i$th cost row. The value and gradient require two log-domain reductions and no inner optimization solve. Write $\operatorname{logsumexp}_i(a_i)=\log\sum_i e^{a_i}$. The first reduction computes $\log Z_i=\operatorname{logsumexp}_j((\tau v_j-C_{ij})/\varepsilon)$ and hence the conjugate value. With $u_i=\log\rho_i^k-\log Z_i$, the second gives
$$
\nabla T_\sigma^*(v)_j=\exp\left(\frac{\tau v_j}{\varepsilon}+\operatorname{logsumexp}_i\left(u_i-\frac{C_{ij}}{\varepsilon}\right)\right).
$$
Log-sum-exp is evaluated stably by subtracting the largest entry before exponentiation. For finite costs and positive total mass, the oracle returns positive cell masses with the same total mass as $\rho^k$. Mass conservation gives $\nabla^2T_\sigma^*(v)\one=0$, while $\nabla^2f(v^\star)=[\nabla^2F(x^\star)]^{-1}+\nabla^2T_\sigma^*(v^\star)$ remains positive definite.}

{All methods solve the same discrete JKO problem. Runs stop when $D_F(x_k^{++},x_k)\leq10^{-6}$ and return $x_k^{++}$, whose objective error is certified by Proposition~\ref{prop:acc-primal-dual-gap}. The plots show the different quantity $D_F(x_k,x^\star)$.}

{In one dimension, reference solutions $x^\star$ are computed by a long solver run followed by damped Newton iteration. Their computed gaps satisfy $D_F(x^{++},x)\leq10^{-15}$.} The fixed-$L$ methods use the local curvature $L_\star$ from \eqref{eq:local-relative-curvature}, computed offline. This pointwise value is an experimental parameter, not a verified bound along the trajectory. {Each single-step experiment starts from a random iterate whose entries are independent $U(0.1,1)$ samples normalized to unit mass per species (seed $0$), shared by all methods.}

We vary $\tau$. Larger values weaken the transport regularizer and reduce the entropic weight $\varepsilon/\tau$. In the reported experiments, they also increase $L_\star$, testing the methods over a range of curvatures. The minimum mass that must be redistributed is $\mathrm{TV}(\rho^k,x^\star)=\tfrac12\|x^\star-\rho^k\|_1$.

\subsection{A granular-media energy}\label{sec:num-granular}

We use the confinement--diffusion--interaction energy \cite{McCann1997,Otto2001,CHPW2017} on $m$ grid points $x_j=-R+(j-1)h$, $j=1,\ldots,m$, with $h=2R/(m-1)$:
\begin{equation}\label{eq:num-granular-F}
F(\rho)=\langle V,\rho\rangle+h\sum_{j=1}^m U(\rho_j/h)
+\tfrac12\lambda\,\rho^\top W\rho,
\end{equation}
with $V(x)=x^2/2$, $U(s)=s^2$, and the Gaussian kernel
\begin{equation}\label{eq:num-granular-W}
W_{jl}=\exp\Bigl(-\frac{(x_j-x_l)^2}{2\sigma_W^2}\Bigr).
\end{equation}
The associated Wasserstein flow is $\partial_t\rho=\nabla\cdot(\rho\nabla(V+U'(\rho)+\lambda W*\rho))$. We take $m=128$, $R=6$, $\varepsilon=10^{-3}$, and $\lambda=\sigma_W=1$. The transport reference is a unit-mass Gaussian bump with mean $2$ and standard deviation $0.5$.

\begin{table}[tbp]
\centering
{\setlength{\tabcolsep}{3pt} \resizebox{0.96\textwidth}{!}{
\begin{tabular}{@{}rrr*{4}{@{\hspace{1.4em}}r@{\hspace{0.45em}}r}@{}}
\toprule
$\tau$ & $L_\star$ & TV & \multicolumn{2}{c}{GCG} & \multicolumn{2}{c}{Adaptive GCG} & \multicolumn{2}{c}{HNAG} & \multicolumn{2}{c}{AHNAG} \\
\cmidrule(lr){4-5}\cmidrule(lr){6-7}\cmidrule(lr){8-9}\cmidrule(lr){10-11}
& & & Calls & Time & Calls & Time & Calls & Time & Calls & Time \\
\midrule
$0.2$ & $242.1$ & $0.153$ & $1583$ & $0.70$ & $252\,(72)$ & $0.10$ & $137$ & $0.06$ & $107\,(8)$ & $0.04$ \\
$0.5$ & $573.0$ & $0.289$ & $3721$ & $1.66$ & $538\,(157)$ & $0.21$ & $210$ & $0.07$ & $168\,(10)$ & $0.06$ \\
$1$ & $1050.2$ & $0.429$ & $6798$ & $3.02$ & $989\,(281)$ & $0.38$ & $285$ & $0.12$ & $214\,(13)$ & $0.07$ \\
$2$ & $1983.9$ & $0.581$ & $12859$ & $5.40$ & $1818\,(519)$ & $0.71$ & $393$ & $0.16$ & $287\,(15)$ & $0.10$ \\
$5$ & $4671.4$ & $0.741$ & $30291$ & $12.44$ & $3674\,(1032)$ & $1.43$ & $602$ & $0.23$ & $429\,(21)$ & $0.15$ \\
$10$ & $9134.4$ & $0.814$ & \multicolumn{2}{c}{NR} & $7148\,(2037)$ & $2.77$ & $841$ & $0.32$ & $556\,(27)$ & $0.20$ \\
$50$ & $44645.3$ & $0.878$ & \multicolumn{2}{c}{NR} & $31467\,(8891)$ & $12.26$ & $1856$ & $0.76$ & $1143\,(41)$ & $0.42$ \\
$100$ & $88985.0$ & $0.886$ & \multicolumn{2}{c}{NR} & \multicolumn{2}{c}{NR} & $2626$ & $1.11$ & $1584\,(57)$ & $0.58$ \\
$200$ & $177329.2$ & $0.890$ & \multicolumn{2}{c}{NR} & \multicolumn{2}{c}{NR} & $3695$ & $1.55$ & $2108\,(68)$ & $0.77$ \\
\midrule
\multicolumn{3}{l}{Fitted exponent $p$} & $0.997$ & --- & $0.925$ & --- & $0.500$ & --- & $0.449$ & --- \\
\bottomrule
\end{tabular}}}
\caption{Granular-media example, $m=128$: oracle calls and CPU time (s) to reach $\mathrm{tol}=10^{-6}$, from the random initial iterate with seed $0$. The fit uses $\text{calls}\sim L_\star^p$ over the rows that reach the tolerance. The comparison slopes are $p=1$ and $p=1/2$, corresponding to nonaccelerated and accelerated scaling; the fits describe these experiments rather than uniform complexity bounds. HNAG uses fixed $L$; AHNAG is Algorithm~\ref{alg:aod}. NR: not reached within $40000$ iterations, with the error still decreasing. For both adaptive methods, parentheses give backtracking calls.}
\label{tab:num-sweep-granular}
\end{table}

\begin{figure}[tbp]
\centering
\includegraphics[width=0.9\textwidth,height = 4.5cm]{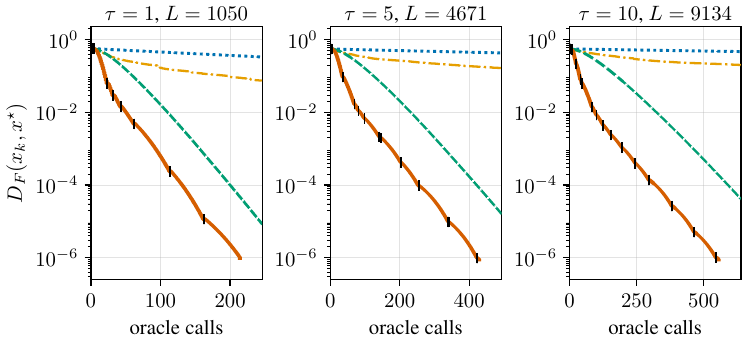}
\caption{{Granular-media example at $m=128$: true error $D_F(x_k,x^\star)$ against oracle calls, at $\tau=1$, $5$ and $10$, from the random initial iterate with seed $0$. Curves: GCG\,\raisebox{-0.1em}{\includegraphics{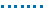}}, adaptive GCG\,\raisebox{-0.1em}{\includegraphics{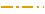}}, HNAG\,\raisebox{-0.1em}{\includegraphics{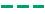}} and AHNAG\,\raisebox{-0.1em}{\includegraphics{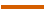}}. Vertical ticks\,\raisebox{-0.1em}{\includegraphics{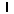}} mark the backtracking calls of AHNAG. Panels share a vertical range and end at $1.15$ times AHNAG's call count, truncating the slower methods.}}
\label{fig:num-granular}
\end{figure}

\begin{figure}[tbp]
\centering
\includegraphics[width=0.85\textwidth,height=4cm]{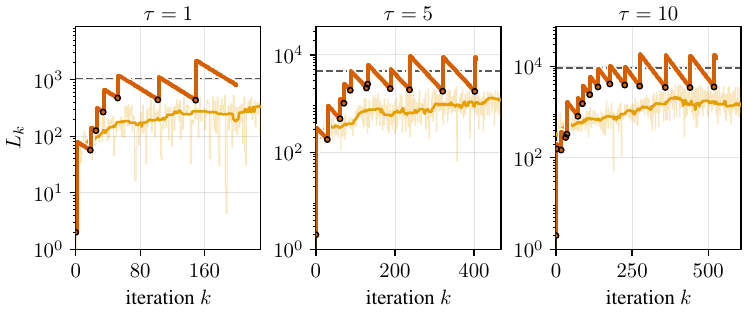}
\caption{Granular-media example at $m=128$: the curvature $L_k$\,\raisebox{-0.1em}{\includegraphics{figures/keys/ahnag}} of AHNAG against the iteration $k$, for the runs of Figure~\ref{fig:num-granular} ($\tau=1$, $5$ and $10$). Circles\,\raisebox{-0.1em}{\includegraphics{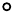}} mark mirror-descent iterations, at the value of $L_k$ before the doubling search. The dashed line\,\raisebox{-0.1em}{\includegraphics{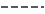}} is $L_\star$. Each panel has its own axes. The faint curve\,\raisebox{-0.1em}{\includegraphics{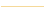}} is adaptive GCG up to $1.15$ times AHNAG's iteration count, and the solid one\,\raisebox{-0.1em}{\includegraphics{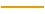}} its running median over $25$ iterations.}
\label{fig:num-Lk-granular}
\end{figure}

The fits in Table~\ref{tab:num-sweep-granular} suggest nearly linear growth in oracle calls for GCG and adaptive GCG and approximately square-root growth for the accelerated methods. The savings thus grow roughly like $\sqrt{L_\star}$. For $\tau\geq100$, adaptive GCG does not reach the tolerance within the cap. AHNAG reaches the tolerance at every tested $\tau$, allowing larger JKO steps within the same iteration budget.

Backtracking accounts for $3.2$--$7.0\%$ of AHNAG's oracle calls, with a smaller share at larger $\tau$. Balanced trials are accepted in $96$--$98\%$ of its iterations, consistent with the positive fraction guaranteed locally by Theorem~\ref{thm:acc-local}. The curvature settles at $0.64$--$0.99\,L_\star$. Figure~\ref{fig:num-granular} shows the error curves.

Figure~\ref{fig:num-Lk-granular} shows how AHNAG adjusts curvature. Starting from $L_0=2$, the first mirror search raises it to the order of $10^2$. Subsequent balanced steps decrease $L_k$ by $\rho$, increasing $\alpha_k=1/\sqrt{L_k}$ and strengthening the certified contraction. Failed trials raise the curvature again, producing a sawtooth pattern near $L_\star$. Adaptive GCG has more frequent fluctuations; its running median reveals a lower curvature scale. 

\subsection{Poisson--Nernst--Planck energy}\label{sec:num-pnp}

The second problem is the two-species Poisson--Nernst--Planck (PNP) model {in the bounded-domain setting of} \cite{KMX2017}. {Let $\rho_+$ and $\rho_-$ be the species cell masses and $\Phi$ the electric potential. On $m$ uniform cells, with $x_j=(j-\tfrac12)h$, $j=1,\ldots,m$, and $h=1/m$, they couple through}
\begin{equation}\label{eq:num-pnp-poisson}
{-\lambda^2\Phi''=(\rho_+-\rho_-)/h,\qquad \Phi'(0)=\Phi'(1)=0,\qquad \textstyle\int_0^1\Phi=0.}
\end{equation}
{The two species have equal mass, so the problem is solvable. Its cell-centered discretization with no-flux ghost values is diagonalized by the discrete cosine transform, so each evaluation of $F$ and $\nabla F$ costs one fast transform pair.} The energy is
\begin{equation}\label{eq:num-pnp-F}
\begin{aligned}
F(\rho_+,\rho_-)={}&\langle\rho_+,\log\rho_+\rangle+\langle\rho_-,\log\rho_-\rangle
+\langle\rho_+,U\rangle+\langle\rho_-,V\rangle\\
&+\tfrac12\langle\rho_+-\rho_-,G(\rho_+-\rho_-)\rangle,
\end{aligned}
\end{equation}
{where $G$ is the discrete Green's operator of \eqref{eq:num-pnp-poisson} and $U,V$ are external potentials evaluated at the cell centers.} The regularizer is the sum of one entropic transport cost per species. {Using $\rho_+$ and $\rho_-$ for the corresponding continuum densities, the system}
$$
{\partial_t\rho_+=\Delta\rho_++\nabla\cdot\bigl(\rho_+\nabla(U+\Phi)\bigr),\qquad
\partial_t\rho_-=\Delta\rho_-+\nabla\cdot\bigl(\rho_-\nabla(V-\Phi)\bigr),}
$$
{together with \eqref{eq:num-pnp-poisson} and no-flux conditions for $\rho_\pm$, is formally a gradient flow in the product Wasserstein metric. Kinderlehrer, Monsaingeon, and Xu \cite{KMX2017} establish this structure and construct weak solutions by the JKO scheme in $\mathbb R^d$, $d\geq3$, and note its extension to bounded domains with no-flux conditions on the densities and the potential and with external potentials satisfying $\partial_\nu U=\partial_\nu V=0$. We take $V=-U$ with $U(x)=10\cos^2(\pi x)$, so $U'(0)=U'(1)=0$. The potential acts like a negatively charged plane at $x=\tfrac12$: it attracts $\rho_+$ and repels $\rho_-$. Kinderlehrer, Monsaingeon, and Xu take $\lambda=1$; here $\lambda$ scales the Debye length.}

The Debye length adds stiffness through $G\sim\lambda^{-2}$. At $\tau=10^{-2}$, decreasing $\lambda^2$ from $10^{-2}$ to $10^{-4}$ raises $L_\star$ from {$9.7$ to $357$}. We use {$m=512$}, $\lambda^2=10^{-4}$, {and} $\varepsilon=1.25\times10^{-3}$. {The transport reference is neutral: both species are the unit-mass Gaussian bump with mean $0.5$ and standard deviation $0.1$.}

\begin{table}[tbp]
\centering
{\setlength{\tabcolsep}{3pt}
\resizebox{0.88\textwidth}{!}{
\begin{tabular}{@{}rr*{4}{@{\hspace{1.4em}}r@{\hspace{0.45em}}r}@{}}
\toprule
$\tau$ & $L_\star$ & \multicolumn{2}{c}{GCG} & \multicolumn{2}{c}{Adaptive GCG} & \multicolumn{2}{c}{HNAG} & \multicolumn{2}{c}{AHNAG} \\
\cmidrule(lr){3-4}\cmidrule(lr){5-6}\cmidrule(lr){7-8}\cmidrule(lr){9-10}
& & Calls & Time & Calls & Time & Calls & Time & Calls & Time \\
\midrule
$0.01$ & $357.1$ & $2287$ & $19.30$ & $349\,(100)$ & $2.90$ & $195$ & $1.64$ & $192\,(12)$ & $1.59$ \\
$0.02$ & $720.1$ & $3516$ & $29.57$ & $455\,(123)$ & $3.78$ & $257$ & $2.16$ & $224\,(12)$ & $1.86$ \\
$0.05$ & $1824.4$ & $5579$ & $46.42$ & $720\,(202)$ & $5.96$ & $378$ & $3.14$ & $314\,(16)$ & $2.60$ \\
$0.1$ & $3683.0$ & $7327$ & $59.74$ & $1124\,(315)$ & $9.21$ & $506$ & $4.11$ & $418\,(21)$ & $3.40$ \\
$0.2$ & $7416.9$ & \multicolumn{2}{c}{NR} & $2014\,(570)$ & $16.47$ & $658$ & $5.36$ & $576\,(27)$ & $4.67$ \\
$0.5$ & $18638.8$ & \multicolumn{2}{c}{NR} & $3413\,(973)$ & $28.07$ & $993$ & $8.12$ & $820\,(36)$ & $6.68$ \\
\midrule
\multicolumn{2}{l}{Fitted exponent $p$} & $0.499$ & --- & $0.593$ & --- & $0.410$ & --- & $0.378$ & --- \\
\bottomrule
\end{tabular}}}
\caption{PNP example, $m=512$ and $\lambda^2=10^{-4}$: oracle calls and CPU time (s) for one proximal step at $\mathrm{tol}=10^{-6}$, from the random initial iterate with seed $0$. HNAG uses fixed $L$; AHNAG is Algorithm~\ref{alg:aod}. NR: not reached within $8000$ iterations, with the error still decreasing. For AHNAG and adaptive GCG, the number in parentheses counts backtracking calls. The fit uses $\mathrm{calls}\sim L_\star^p$ over the converged rows: four for GCG and six for each other method.}
\label{tab:num-pnp-single}
\end{table}

\begin{figure}[tbp]
\centering
\includegraphics[width=0.9\textwidth,height=4.5cm]{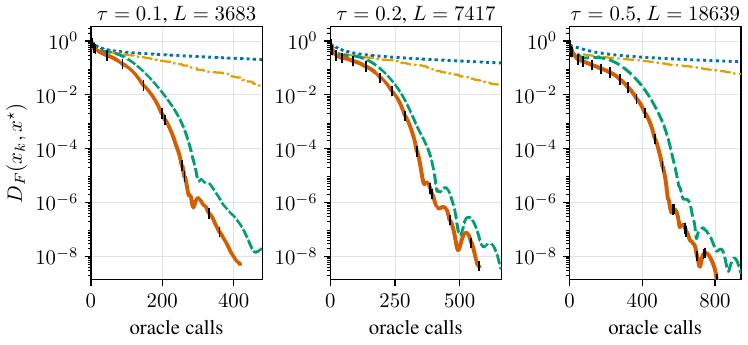}
\caption{{PNP example at $m=512$: true error $D_F(x_k,x^\star)$ against oracle calls, at $\tau=0.1$, $0.2$ and $0.5$, from the random initial iterate with seed $0$. Curves: GCG\,\raisebox{-0.1em}{\includegraphics{figures/keys/gcg}}, adaptive GCG\,\raisebox{-0.1em}{\includegraphics{figures/keys/agcg}}, HNAG\,\raisebox{-0.1em}{\includegraphics{figures/keys/hnag}} and AHNAG\,\raisebox{-0.1em}{\includegraphics{figures/keys/ahnag}}. Vertical ticks\,\raisebox{-0.1em}{\includegraphics{figures/keys/tick}} mark AHNAG's backtracking calls. Panels share a vertical range and end at $1.15$ times AHNAG's call count, truncating the slower methods.}}
\label{fig:num-pnp}
\end{figure}

In Table~\ref{tab:num-pnp-single}, AHNAG uses fewer calls than adaptive GCG and improves on HNAG without the offline curvature. The fitted exponents are below the comparison values $1$ and $1/2$. The smallest nontrivial eigenvalue of the local relative Hessian in \eqref{eq:local-relative-curvature} grows from $1.4$ to $9.8$ as $\tau$ increases. Thus its condition number grows more slowly than $L_\star$.

Backtracking accounts for $4.4$--$6.0\%$ of AHNAG's oracle calls, and $97\%$ of iterations accept the balanced trial. This finite-run observation is consistent with the local acceptance mechanism. The curvature settles at $0.70$--$0.80\,L_\star$.
The error curves in Figure~\ref{fig:num-pnp} need not be monotone: the contraction estimate applies to the combined state energy $\mathcal E_k$.

{We repeat this experiment on the box $\Omega=[0,1]^3$, with the same charged-center potential: $V=-U$ and $U(x,y,z)=S\bigl(\cos^2(\pi x)+\cos^2(\pi y)+\cos^2(\pi z)\bigr)$, $S=30$, so $\partial_\nu U=\partial_\nu V=0$ on $\partial\Omega$. The energy, discretization, and transport oracle are the three-dimensional analogues of \eqref{eq:num-pnp-poisson}--\eqref{eq:num-pnp-F}: the cell-centered Laplacian with homogeneous Neumann data is diagonalized by the three-dimensional discrete cosine transform, so evaluating $F$ and $\nabla F$ costs one fast transform pair. The Gibbs kernel of the transport cost is a Kronecker product of three one-dimensional kernels. Thus each log-sum-exp reduction in the oracle costs $O(n^4)$ instead of $O(n^6)$ without forming an $n^3\times n^3$ matrix.

We use $n=32$ uniform cells per axis, so each species has $32^3=32768$ unknowns, with $\lambda^2=3\times10^{-4}$, $\varepsilon=2\times10^{-3}$, and $\tau=0.05$, $0.2$, and $1$. The transport reference is neutral: both species are the unit-mass Gaussian bump centered at $(0.5,0.5,0.5)$ with standard deviation $0.12$. The starting iterate (seed $0$), the stopping test, and the parameters of AHNAG are as above.

The reference solution $x^\star$ is computed by an accelerated run to $D_F(x^{++},x)\leq10^{-12}$. To avoid forming a dense Hessian, we estimate $L_\star$ in \eqref{eq:local-relative-curvature} by matrix-free power iteration.}

\begin{table}[tbp]
\centering
{\setlength{\tabcolsep}{3pt} \resizebox{0.85\textwidth}{!}{
\begin{tabular}{@{}rr*{4}{@{\hspace{1.4em}}r@{\hspace{0.45em}}r}@{}}
\toprule
$\tau$ & $L_\star$ & \multicolumn{2}{c}{GCG} & \multicolumn{2}{c}{Adaptive GCG} & \multicolumn{2}{c}{HNAG} & \multicolumn{2}{c}{AHNAG} \\
\cmidrule(lr){3-4}\cmidrule(lr){5-6}\cmidrule(lr){7-8}\cmidrule(lr){9-10}
& & Calls & Time & Calls & Time & Calls & Time & Calls & Time \\
\midrule
$0.05$ & $1104.5$ & \multicolumn{2}{c}{NR} & $1252\,(312)$ & $135.32$ & \multicolumn{2}{c}{NR} & $613\,(21)$ & $66.01$ \\
$0.2$ & $2465.2$ & \multicolumn{2}{c}{NR} & $1968\,(519)$ & $212.39$ & \multicolumn{2}{c}{NR} & $861\,(26)$ & $91.99$ \\
$1$ & $10382.9$ & \multicolumn{2}{c}{NR} & \multicolumn{2}{c}{NR} & \multicolumn{2}{c}{NR} & $2076\,(54)$ & $221.41$ \\
\bottomrule
\end{tabular}}}
\caption{Three-dimensional PNP example, $n=32$, $\lambda^2=3\times10^{-4}$: oracle calls and CPU time (s) for one proximal step at $\mathrm{tol}=10^{-6}$, from the random initial iterate with seed $0$. NR: not reached within $2000$ iterations. For AHNAG and adaptive GCG, the number in parentheses counts backtracking calls.}
\label{tab:pnp3d}
\end{table}

\begin{figure}[tbp]
\centering
\includegraphics[width=0.9\textwidth,height=4.7cm]{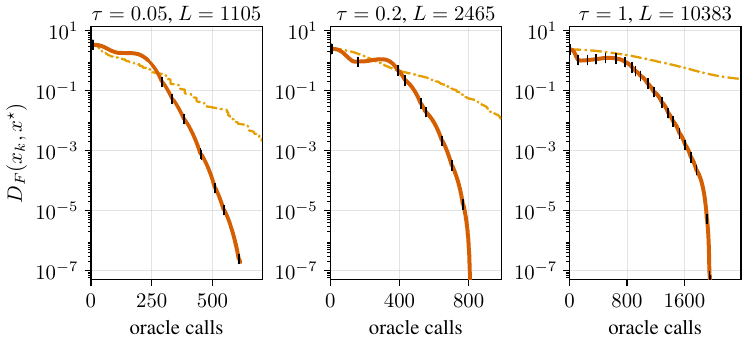}
\caption{{Three-dimensional PNP example, $n=32$: true error $D_F(x_k,x^\star)$ against oracle calls, at $\tau=0.05$, $0.2$ and $1$, from the same random starting iterate as Table~\ref{tab:pnp3d}. GCG and HNAG are omitted, since neither reaches the tolerance at any $\tau$ (Table~\ref{tab:pnp3d}). Curves: adaptive GCG\,\raisebox{-0.1em}{\includegraphics{figures/keys/agcg}} and AHNAG\,\raisebox{-0.1em}{\includegraphics{figures/keys/ahnag}}. Vertical ticks\,\raisebox{-0.1em}{\includegraphics{figures/keys/tick}} mark AHNAG's backtracking calls. Panels share a vertical range and end at $1.15$ times AHNAG's call count.}}
\label{fig:pnp3d}
\end{figure}

\begin{figure}[tbp]
\centering
\includegraphics[width=0.9\textwidth,height = 4.25cm]{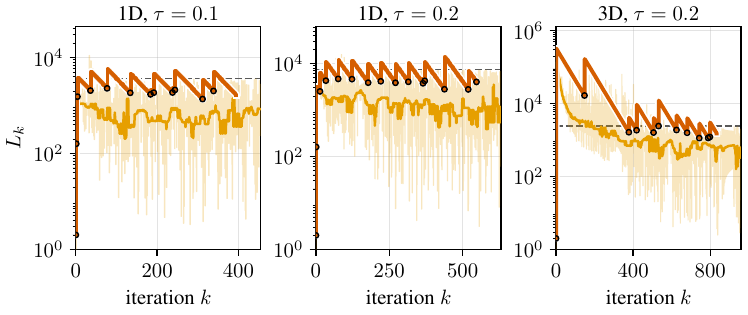}
\caption{PNP example: the curvature $L_k$\,\raisebox{-0.1em}{\includegraphics{figures/keys/ahnag}} of AHNAG against the iteration $k$, in one dimension ($\tau=0.1$ and $0.2$, the runs of Figure~\ref{fig:num-pnp}) and three dimensions ($\tau=0.2$, the run of Figure~\ref{fig:pnp3d}). Circles\,\raisebox{-0.1em}{\includegraphics{figures/keys/md}} mark mirror-descent iterations, at the value of $L_k$ before the doubling search. The dashed line\,\raisebox{-0.1em}{\includegraphics{figures/keys/lstar}} is $L_\star$. Each panel has its own axes. The faint curve\,\raisebox{-0.1em}{\includegraphics{figures/keys/agcg-faint}} is adaptive GCG up to $1.15$ times AHNAG's iteration count, and the solid one\,\raisebox{-0.1em}{\includegraphics{figures/keys/agcg-median}} its running median over $25$ iterations.}
\label{fig:num-Lk-pnp}
\end{figure}

Table~\ref{tab:pnp3d} and Figure~\ref{fig:pnp3d} show the results. With fixed $L=L_\star$, the error grows during the reported runs for HNAG at every $\tau$ and for GCG at $\tau=0.05$. Neither fixed-step method reaches the tolerance within the cap. Adaptive GCG reaches it at $\tau=0.05$ and $0.2$, while AHNAG reaches it at all three values. At the two smaller steps, AHNAG uses about half the oracle calls of adaptive GCG. Backtracking accounts for $2.6$--$3.3\%$ of AHNAG's oracle calls, and $99\%$ of its iterations accept the balanced trial. Its error stays at or below its initial value throughout.  Fig.~\ref{fig:num-Lk-pnp} shows the change of $L_k$ for several $\tau$ in 1-D or 3-D. 

\subsection{Multi-step flows}\label{sec:num-flow}

{We run {$15$} JKO steps with {$(\tau,m)=(5,128)$} for granular media, {$(0.05,512)$} for PNP, and $(\tau,n)=(0.05,32)$ for three-dimensional PNP, where HNAG is omitted since it does not reach the tolerance within the cap at this $\tau$ (Table~\ref{tab:pnp3d}). Each step uses the previous step's final iterate $x_k$ as its transport reference and initial iterate; only the first initial iterate is random.} {Figure~\ref{fig:num-flow} shows the cumulative CPU time; HNAG uses the estimate of $L_\star$ from the first step.} {Each method computes its own flow; terminal densities agree within {$2.6\times10^{-4}$} in $\ell^1$, measuring solver agreement rather than discretization accuracy.}

\begin{figure}[tbp]
\centering
\includegraphics[width=0.9\textwidth,height=4.5cm]{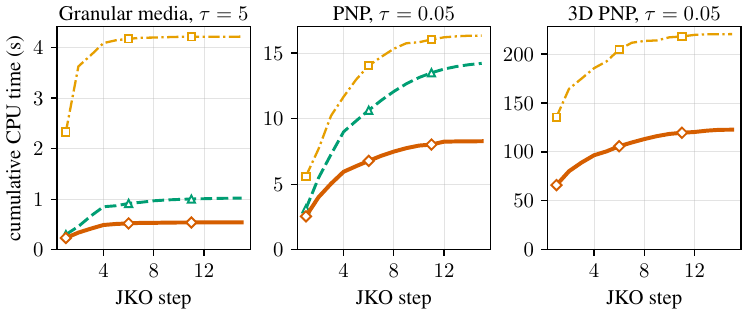}
\caption{{Cumulative CPU time (s) across the warm-started flows, granular media (left, $m=128$), PNP (middle, $m=512$) and three-dimensional PNP (right, the setting of Table~\ref{tab:pnp3d}; HNAG is omitted, since it does not reach the tolerance there), from the random initial iterate with seed $0$. Curves: adaptive GCG\,\raisebox{-0.1em}{\includegraphics{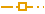}}, HNAG\,\raisebox{-0.1em}{\includegraphics{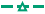}} and AHNAG\,\raisebox{-0.1em}{\includegraphics{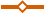}}.}}
\label{fig:num-flow}
\end{figure}

Most of the cumulative time in Figure~\ref{fig:num-flow} is spent on the early JKO steps. Reusing the previous output as the initial guess reduces the later inner solves. AHNAG has the lowest cumulative time in all three examples, and the savings over adaptive GCG persist across the flow. AHNAG also solves larger JKO steps within the tested budgets: at $\tau=1$ in Table~\ref{tab:pnp3d}, AHNAG reaches the tolerance, whereas adaptive GCG does not within the iteration cap.

\section{Conclusion and future work}\label{sec:conclusion}

We proposed a practical adaptive acceleration of generalized conditional gradient methods based on a computable primal--dual certificate. The method accepts a balanced accelerated trial when the certificate passes and otherwise takes a mirror-descent step. Under the stated assumptions, every iteration contracts the Lyapunov function, giving global linear convergence; near the solution, a positive fraction of balanced steps gives a local accelerated rate. The granular-media and Poisson--Nernst--Planck experiments show fewer oracle calls than adaptive GCG, including in three dimensions.

Future work will focus on large-scale applications, where the regularized linear oracle accounts for most of the computational cost. Efficient parallel implementations and approximate oracle evaluations could extend the method to finer discretizations and larger transport problems. This also calls for extending the computable certificate and convergence analysis to controlled oracle errors.

\section*{Acknowledgments}
The authors used OpenAI Codex to assist with writing and with developing and checking the proofs. The authors are responsible for the mathematical results and the final manuscript.

\bibliographystyle{siamplain}
\bibliography{references,aamd_references}
\end{document}